\documentclass{article}

\usepackage{fichierbase_english}
\usepackage[margin=2cm]{geometry}

\renewcommand{\k}{\mathbf k}

\newcommand{\h}{\mathfrak h}
\newcommand{\U}{\mathcal U}
\newcommand{\x}{\mathbf x}
\newcommand{\y}{\mathbf y}
\newcommand{\z}{\mathbf z}
\newcommand{\g}{\mathfrak g}
\newcommand{\X}{\mathbf X}
\newcommand{\Y}{\mathbf Y}
\newcommand{\GL}{\mathrm{GL}}
\newcommand{\Chi}{\mathcal X}

\DeclareMathOperator{\Diff}{Diff}
\DeclareMathOperator{\dl}{dl}
\DeclareMathOperator{\nilp}{nilp}
\DeclareMathOperator{\Char}{char}
\DeclareMathOperator{\Inv}{Inv}
\DeclareMathOperator{\ad}{ad}

\DeclareMathOperator{\Tri}{Tri}
\DeclareMathOperator{\vdl}{vdl}
\DeclareMathOperator{\Ad}{Ad}
\DeclareMathOperator{\Bir}{Bir}
\DeclareMathOperator{\Specmax}{Specmax}
\DeclareMathOperator{\Spec}{Spec}
\DeclareMathOperator{\Vect}{Vect}
\newcommand{\ent}[1]{\left\lfloor #1 \right\rfloor}

\begin{document}

\title{Actions of nilpotent groups on complex algebraic varieties}

\author{Marc Abboud}

\maketitle

\abstract{We study nilpotent groups acting faithfully on complex algebraic varieties. We use a method of base
change. For finite $p$-groups, we go from $\k$, a number field, to a finite field in order to use counting lemmas. We show that a finite
$p$-group of polynomial automorphisms of $\k^d$ is isomorphic to a
subgroup of $\GL_d(\k)$. For infinite groups, we go from $\C$ to $\Z_p$ and use $p$-adic analytic tools and the theory
of $p$-adic Lie groups. We show that a finitely generated nilpotent group $H$ acting faithfully on a complex quasiprojective
variety $X$ of dimension $d$ can be embedded into a $p$-adic Lie group acting faithfully and analytically on $\Z_p^d$; we deduce
that $d$ is larger than the virtual derived length of $H$.
}
\setcounter{tocdepth}{1}
\tableofcontents

\section{Introduction}

\subsection{Minkowski's bound for polynomial automorphisms.}\label{par:Minkowski_Schur}

\paragraph{Rational numbers.--} Let $p$ be a prime. A finite \emph{$p$-group} is a group of size $p^\alpha$ for some integer $\alpha \geq 0$.
For $d\in \Z_+$, define $M_\Q(d,p)$ to be the integer
\[ M_\Q(d,p) = \ent{\frac{d}{p-1}} + \ent{\frac{d}{p(p-1)}} + \ent{\frac{d}{p^2 (p-1)}} + \cdots
\]
(Here $M$ stands for Minkowski).
Let $v_p$ be the $p$-adic valuation; then $M_\Q(d,p) = \ent{\frac{d}{p-1}} + v_p \left( \ent{\frac{d}{p-1}} !
\right)$.

\begin{thm}[Minkowski 1887, see \cite{SerreBoundsOrder}]\label{MinkowskiBoundLinear} Let $d$ be a natural
number and let $p$ be a prime. If $G$ is a finite $p$-subgroup of ${\mathrm{\GL}}_d (\Q)$, then $v_p(\vert G\vert)\leq M_\Q(d,p)$,  and this upper bound is
optimal: there are groups of order $p^{M_\Q(d,p)}$ in ${\mathrm{\GL}}_d(\Q)$.  \end{thm}

\paragraph{Number fields.--} Schur extended Minkowski's result to the case of number fields. To state Schur's result,
let us introduce some notation for cyclotomic extensions. Consider a number field $\k$ and fix an
algebraic closure ${\overline{\k}}$ of $\k$. Denote by $z_a\in {\overline{\k}}$ any primitive $a$-th root of unity, for
$a$ any positive integer; for instance $z_4={\mathsf{i}}$,
a square root of $-1$.
\begin{itemize}
\item If $p \geq 3$, set $t(\k;p) = [\k(z_p) : \k]$ and let $m(\k;p)$
be the maximal integer $a$ such that $\k(z_p)$ contains $z_{p^a}$; note that $m(\k; p)$ is finite because $\k$ is a finite extension of $\Q$. Then, define
\[
M_\k(d,p) := m(\k; p) \cdot \ent{\frac{d}{t(\k;p)}} + \ent{\frac{d}{p \cdot t(\k;p)}} +
\ent{\frac{d}{p^2t(\k;p)}} + \cdots .
\]
\item  If $p=2$, set $t(\k; 2) =[\k(z_4): \k]$  and let $m(\k; 2)$ be the largest integer $a$ such that $z_{2^a} \in \k(z_4)$. Define \[
M_\k (d,2) = d + (m(\k; 2)-1) \ent{\frac{d}{t(\k;2)}} +  \ent{\frac{d}{2t(\k;2)}} +
\ent{\frac{d}{4t(\k;2)}} + \cdots .
\]
\end{itemize}
This definition is consistent with the definition of $M_\Q (d,p)$ given above.

\begin{thm}[\cite{schur1973klasse}, \cite{SerreBoundsOrder}]\label{SchurBound} Let $d$ be a natural
number, and let $p$ be a prime.
 If $G$ is a finite $p$-subgroup of ${\mathrm{\GL}}_d (\k)$  then $v_p(\vert G\vert)\leq M_\k(d,p)$ and
 this bound is optimal.
 \end{thm}

 It is not difficult to find a subgroup $G \subset \mathrm{\GL}_d (\k)$ such that $\abs G =
 p^{M(d,p)}$. We recall how to do so in Proposition \ref{PropBorneOptimale}.

\paragraph{Polynomial automorphisms.--}
Our first goal is to extend the theorem of Minkowski and Schur to an algebraic,
but nonlinear context. Let $\Aut(\A_\k^d)$ be the group of polynomial automorphisms of the affine space
$\A^d$, over some number field $\k$. This group contains $\GL_d (\k)$ but it
is much more complicated. Surprisingly, we are able to show that the Minkowski-Schur bound still
holds for subgroups of $\Aut (\A_\k^d)$ and in fact the same finite subgroups appear.

\begin{bigtheorem}\label{SchurBoundPolynomial}
Let $\k$ be a number field, let $d$ a natural number, and let $p \geq 3$ be a prime. If $G$ is a finite
$p$-subgroup of $\Aut(\A_\k^d)$, then there exists a group embedding $ G \hookrightarrow \GL_d(\k).$
In particular, Schur's bound still holds:
\[
v_p (\abs G) \leq M_\k(d,p),
\]
and this bound is optimal.

\end{bigtheorem}

 The proof first shows the bound on the cardinal of the group $G$ and we then find the group embedding $G \hookrightarrow
 \GL_d(\k)$ using a Sylow argument.

 \begin{rmq} \label{remark:CasPImpairPasOptimal} The case $p=2$ is
  also dealt with in Section 2. But we don't get an optimal bound. For example for $p=2$ and $\k = \Q$, we show that any
  $2$-subgroup $G$ of $\Aut(\A_\Q^d)$ can be embedded into $\GL_d(\Q(z_4))$ and therefore satisfies
  $v_2(\abs G) \leq M_\Q(d,2) + \ent{\frac{d}{2}}$. More
  precisely, Proposition \ref{ImageCaractereCyclotomique} defines three cases $(a), (b)$ and $(c)$ when
  $p=2$. We get an embedding into $\GL_d(\k)$
  in case $(a)$ and $(b)$ (this is the case for example if $\k$ contains $z_4$), but in case $(c)$ we can
  only get an embedding of $G$ into $\GL_d(\k(z_4))$ and therefore we get the bound $ v_2(\abs G) \leq
  M_\k(d,2) + \ent{\frac{d}{2}} = M_{\k(z_4)}(d,2) $. See Theorem \ref{BigThmSchurBoundPolynomial} page
  \pageref{BigThmSchurBoundPolynomial} for the general statement.

  In fact, Theorem \ref{SchurBoundPolynomial} still holds when $\k$ is a finitely generated
 field over $\Q$ but the proof is less intuitive so we will show the proof for $\k$ a number field and
 explain how to extend it to finitely generated field over $\Q$ in Remark \ref{RmqFinitelyGeneratedField}.
 We then state the complete theorem for finitely generated fields over $\Q$ in Theorem
 \ref{BigThMSchurBoundPolynomialGeneralCase} page \pageref{BigThMSchurBoundPolynomialGeneralCase}.

 \end{rmq}

Our method of proof follows \cite{SerreBoundsOrder}, in which Serre bounds the order of the finite subgroups
of ${\mathrm{H}}(\k)$, for ${\mathrm{H}}$ a semi-simple algebraic group; the phenomenon mentioned in Remark
\ref{remark:CasPImpairPasOptimal} also appears for such groups ${\mathrm{H}}$. The general idea is to embed $G$ into a group of
linear automorphisms over a finite field, study the finite field case, and use cyclotomic characters to find the
optimal bound yield by this method.

\paragraph{Birational transformations.--}
The problem of the existence of uniform bounds on the size of finite $p$-groups or finite simple groups in infinite
dimensional groups such as $\Aut(\A^d)$ or $\Bir(\A^d)$ has been studied extensively during the last decade
(see~\cite{Serre09}). For an arbitrary complex projective variety $X$, one cannot expect uniform
bounds that would only depend on the dimension of $X$, since
every finite group is the  group of automorphisms of a complex projective curve (see~\cite{Greenberg}).
But precise results have been obtained when $X$ is rationally connected. Recently, Jinsong Xu showed the following
optimal result:
{\sl{Let $d$ be a natural number and let $p$ be a prime $>d+1$.
If $X$ is a rationally connected variety of dimension $d$ over an
algebraically closed field of characteristic $0$, and  $G$ is a finite $p$-subgroup of $\Bir(X)$, then G is abelian
and its rank is at most $d$}} (see~\cite{Jinsong_Xu:CRAS_2020}). Results of this type were first shown by
Prokhorov, Shramov and Birkar in \cite{prokhorov_shramov_2014} for birational transformations of any varieties and
improvements were made for rationally connected varieties in  \cite{prokhorov2016jordan}.

These results are deeper than our Theorem \ref{SchurBoundPolynomial}, but our contribution has
a few advantages: it
may serve as an introduction to the work of Prokhorov and Shramov, the techniques are more elementary, the precise
bound we obtain illustrates the interplay between the arithmetic of the field $\k$ and the size of the group, and the proof
shows why the upper bound of Minkowski and Schur is still valid in $\Aut(\A_\k^d)$.

\begin{rmq}
The results of Prokhorov and Shramov rely on the  BAB
conjecture, which was proved by Birkar in \cite{birkar2016singularities}. The result of J.
Xu relies on the work of Haution on
equivariant cohomology and fixed points of finite groups (see~\cite{haution_2019}).
\end{rmq}

\subsection{A bound for the action of finitely generated nilpotent groups}

\subsubsection{Nilpotent and solvable groups}\label{par:nilpotent_and_solvable}

Let~$H$~be a group.  If~$a,b \in H$, we denote by~$[a,b] := ab \inv a \inv b$~their commutator.   If~$H_1,H_2$~are two
subgroups of~$H$, then we denote by~$H_1 H_2$~the subgroup generated by the set~$\{h_1 h_2: h_1 \in H_1, h_2 \in
H_2\}$~and by $[H_1,H_2]$ the subgroup generated by the set~$\{[h_1, h_2]: h_1 \in H_1, h_2 \in H_2\}$.
The lower central (resp. derived) series is defined by~$D^{0}(H) = H$~(resp.~$D_0(H) = H$)~and~$D^{i+1}(H)
= [H, D^{i}(H)]$~(resp.  $D_{i+1}(H) = [D_i(H), D_i (H)]$). A group $H$ is  \emph{nilpotent} (resp.
\emph{solvable}) when there exists an integer $k$ such that $D^k(H)  = {1}$ (resp. $D_k(H)  = {1}$).

If $H$ is nilpotent, its {\emph{nilpotency
class}} $\nilp(H)$~is the lowest integer such that $D^k(H)  = {1}$.
For a solvable group
$H$, denote by $\dl (H)$ its derived length, that is the least integer $k$ such that $D_k(H)  = {1}$. The \emph{virtual
derived length} is the minimum of $\dl(H_0)$ over finite index subgroups $H_0$ of $H$. Similar definitions
and notation will be used for Lie algebras.

\subsubsection{Upper bounds on the virtual derived length}

Finite $p$-groups are nilpotent. We now look at infinite, finitely generated nilpotent groups, and their actions by
automorphisms and birational transformations.
In \cite{cantat2014algebraic}, Cantat and Xie used $p$-adic analysis to give information on group actions on complex
algebraic varieties by birational transformations, and sketched the proof of the following result.

\begin{bigtheorem}\label{BoundNilpotentGroups}
  Let $H$ be a finitely generated nilpotent
  group acting faithfully on a quasi-projective variety $X$ by algebraic automorphisms over a field of characteristic
zero. Then, \[ \vdl(H)\leq \dim X  \] where $\vdl(H)$ is the \emph{virtual derived length} of $H$.
Furthermore, this bound is optimal.
\end{bigtheorem}

Another goal of this paper is to give a complete proof of this result. Again, the main idea is to replace the
initial field of definition by another one, here $\Q_p$, and in fact by $\Z_p$, for a suitable prime $p$.
Then, the initial action of the discrete group $H$ will be extended to an analytic action of a $p$-adic
Lie group over $\Z_p^{\dim X}$, so that
tools from $p$-adic analysis will be available, in particular $p$-adic analytic vector fields and $p$-adic Lie algebras. Thus,
Theorem \ref{BoundNilpotentGroups} will follow from a similar theorem we prove over $\Z_p$. Section
\ref{SecPAdicAnalysis} is dedicated to the construction of $p$-adic analytic tools needed for the proof of Theorem
\ref{BoundNilpotentGroups} such as infinite dimensional $p$-adic Lie groups or Tate-analytic diffeomorphisms and Section
\ref{SecFinitelyGeneratedNilpotentGroups} is dedicated to the proof of Theorem \ref{BoundNilpotentGroups}.

\section{Finite $p$-groups}

\subsection{Preliminaries}

\paragraph{Primes and $p$-adic numbers} In the rest of the article, $p$ is a prime unless mentioned otherwise, $\Z_p$
denote the ring of $p$-adic integers and $\Q_p$ is the fraction field of $\Z_p$. Recall that Dirichlet's
theorem states for any integers $a,n$ such that $gcd(a,n) =1$, there is an infinite amount of prime numbers
$\ell$ such that $\ell = a \mod n$.

\paragraph{Maximal ideals and reduction} If $q$ is a power of a prime, we denote by $\F_q$ the field with $q$ elements.
  Let $A$ be a finitely generated $\Z$-algebra. Then for every maximal ideal $\m \subset A$, $A / \m$ is a finite
  field. This comes from the Nullstellensatz for Jacobson rings which is proven in \cite{bourbaki2007algebre}, chapter
  5, $\S 3$, theorem 3 of section 4.

\subsection{Groups of linear transformations over $\Q$}

To warm up, let us prove the theorem of Minkowski. For a ring $A$, we denote by $A^{\times}$ its subgroup of invertible
elements; for any prime $p$ the group $(\Z / p^2 \Z)^{\times}$ is cyclic.

\begin{prop}\label{PropEmbeddingFinite}
  Let $G$ be a finite subgroup of $\GL_d (\Q)$. For any prime $\ell$ large enough there exists an injective homomorphism
  $G \hookrightarrow \GL_d (\F_\ell)$.
\end{prop}

\begin{proof}
  Since $G$ is finite, there exists an integer $N$ such that $G \subset \GL_d (\Z[1 / N])$. Now, for each $g \in G
  \setminus \left\{ \id \right\}$ denote by $l(g)$ the largest prime factor that appears in the prime decomposition of the rational
  numbers given by the coefficients of the matrix $g - \id$; denote by $L$ the maximum of the primes $l(g)$. If $\ell >
  \max(N,L)$, the homomorphism of reduction modulo $l$ is defined on $G$ and is injective.
\end{proof}

Thus, if $G \subset \GL_d ( \Q)$ is a finite subgroup, $v_p(\abs G) \leq v_p (\abs{\GL_d(\F_\ell)})$ for any $\ell$
given by Proposition \ref{PropEmbeddingFinite}. We know that

  \begin{equation}
\abs{ \GL_d (\F_\ell)} = \ell^{d(d-1)/2} \prod_{i=1}^{d-1} \left( \ell^i -1 \right).
    \label{EqCardinalGLd}
  \end{equation}

for any prime $\ell$. Let us compute the $p$-adic valuation of such a product.

\begin{lemme} \label{CalculValuationsP-Adiques}
  Suppose $p\neq 2$ and let $\ell$ be a generator of $(\Z / p^2 \Z)^\times$.
  \begin{enumerate}
    \item If $p$ divides $\ell^i -1$ then $p-1$ divides $i$;
    \item If $p-1$ divides $i$ then $v_p(\ell^i-1) = 1 + v_p(i)$.
  \end{enumerate}
\end{lemme}

\begin{proof}
  Suppose $p$ divides $\ell^i-1$. Note that $\ell^{ip}-1 = (\ell^i -1) \sum_{j=0}^{p-1} \ell^{ij}$; since
  $\ell^i \equiv 1
  \mod p$, we have $\sum_{j=0}^{p-1} \ell^{ij} \equiv 0 \mod p$, and then $\ell^{ip} \equiv 1 \mod p^2$.
  Since $\ell$ is of
  order $p(p-1)$ in $(\Z / p^2 \Z)^{\times}$, we have that $p(p-1)$ divides $ip$
  therefore, $p-1$ divides $i$, which proves the first assertion.

  We prove assertion 2 by induction on $v_p(i)$.  To initialize the induction assume $v_p(i) =0$. Then
  $p$ and therefore $p(p-1)$ do not divide $i$; thus $\ell^i \not \equiv 1 \mod p^2$ because $\ell$ is of order $p(p-1)$.
  Thus, $v_p(\ell^i -1) =1$. Now suppose the assertion true for $v_p(i) = k$ with $k \geq 0$ and suppose $v_p(i) = k+1$.
  Write $i = (p-1) p^{k+1} m$ with $m$ not divisible by $p$ and suppose the result true for $v_p(i) = k$.
  Let $s:= \ell^{(p-1)m}$, then
  \begin{align*}
    \ell^{i} -1 = s^{p^{k+1}} -1 = (s^{p^k}-1) \sum_{j=0}^{p-1} s^{j p^k}.
  \end{align*}
  By induction, $s^{p^k}$ is of the form $s^{p^k} = 1 + u p^{k+1}$ where $u$ is an integer not divisible by $p$. Therefore,
  for all $1 \leq j \leq p-1, s^{jp^k} = 1 + j p^{k+1} u + v_jp^2$ where $v_j$ is some integer.
  , therefore we can write
  \[ \sum_{j=0}^{p-1} s^{jp^k} = p + p^{k+1} \frac{p(p-1)}{2} u + p^2V = p \left(1 + p^{k+1} \frac{p-1}{2} u + pV \right) \]
  where $V = \sum v_j$. Since $p$ is odd, $\frac{p-1}{2}$ is an integer and this sum has $p$-adic valuation $1$ since
  $k+1 \geq 1$.
  \[
  s^{p^{k+1}} -1 = (s^{p^k} -1) \cdot p \left(1 + p^k\sum_{j=0}^{p-1} u_j \right) \]
  since $k \geq 1$, we get $v_p (s^{p^{k+1}} -1) = 1 + v_p (s^k -1 ) = 1 + (k+1)$.
\end{proof}

Equation \eqref{EqCardinalGLd} and Lemma \ref{CalculValuationsP-Adiques} provide the following corollary.

\begin{cor}
  Let $d$ be an integer, let $p$ be an odd prime, and let $\ell$ be a prime whose image in $(\Z / p^2
  \Z)^{\times}$ is a generator. Then
  \[ v_p (\GL_d(\F_\ell)) = M_\Q(d,p). \]
  This proves also the fact that Theorem \ref{MinkowskiBoundLinear} "is optimal for $\GL_d(\F_\ell)$" by Sylow.
\end{cor}

To prove Theorem \ref{MinkowskiBoundLinear}, consider a finite group $G \subset \GL_d (\Q)$, then apply Dirichlet's
theorem and Proposition \ref{PropEmbeddingFinite} to embed $G$ in $\GL_d(\F_\ell)$ for some prime generator $\ell$ of
$(\Z / p^2 \Z)^*$. The corollary gives the desired upper bound.

\begin{rmq}\label{remarkCasPairLinear}
  The case $p=2$ is also treated by Minkowski and in fact the same bound applies. However the proof is slightly
  different as it is required to embed $G$ into an orthogonal group over a finite field. Indeed, If $\ell$ is an odd prime
  then the best bound one can get is $v_2(\GL_d(\F_l)) \leq M(d,2) + \lfloor d/2 \rfloor$ with equality with the right
  choice of $\ell$ (see Proposition \ref{ConstanteMajoréeParBorneSchur}). To embed a finite group $H$ of
  matrices over $\Q$ into an orthogonal group over a finite field, one just need to look at the positive
  definite bilinear form $\psi := \sum_{h \in H} {}^t \! H H$.  For any prime $\ell$ large enough such that
  $\ell$ does not divide $\det \psi$, the group homomorphism of reduction mod $\ell$ induces an embedding of
  $H$ into an orthogonal group over $\F_l$, however this process does not generalize well when looking at
  polynomial automorphisms (See Remark \ref{remarkCasPairNeMarchPas}).
\end{rmq}

\subsection{The Minkowski's bound for finite groups of polynomial automorphisms with rational coefficients}

To prove Theorem \ref{SchurBoundPolynomial}, we adapt the proof of the Minkowski bound for linear automorphisms.
Actually, to conclude it suffices to show that Proposition \ref{PropEmbeddingFinite} also holds for finite $p$-subgroups of
polynomial automorphisms.

\begin{prop}\label{EmbeddingPolynomialCase}
  Let $d$ be an integer. Let $G$ be a finite $p$-subgroup of $\Aut(\A_\Q^d)$.
  Then, there exists a prime $\ell$ such that
  \begin{enumerate}
    \item $\ell$ is a generator of $(\Z / p^2 \Z)^{\times}$
    \item There is an injective homomorphism $G \hookrightarrow \GL_d (\F_\ell)$
  \end{enumerate}

\end{prop}

\begin{lemme}\label{PointFixeEtDifferentielleInj}
  Let $d$ be an integer and $p$ a prime. Let $F$ be a finite field with $\Char (F) \neq p$. Let $G$ be a finite
  subgroup of $\Aut (\A_F^d)$ of order $p^\alpha$. Then $G$ has a fixed point $x_0 \in \A^d (F) = F^d$ and the homomorphism
  \[
  \begin{array}{lrcl}
    \Phi:& G & \longrightarrow & \GL_d (F)\\
          & g & \longmapsto & D_{x_0} g
  \end{array}
\]

  is injective.
\end{lemme}

\begin{proof}
  The group $G$ acts on $F^d$ which is of size $\abs F^d$. Since $\abs G = p^\alpha$ and $p$ does not divide $\abs F$,
  the class equations gives the existence of at least one trivial $G$-orbit in $F^d$; hence, the existence of a fixed
  point $x_0 \in F^d$.

  Up to a translation we can suppose that $x_0 = 0$.  Now to show the injectivity of $\Phi$. Take $g$ in $G$ such that
  $D_{0} g = \id$, then \[ g(x_1,\cdots, x_d) = g(\x) = \id + \sum_{j \geq 2} A_j (\x) \]
  where $A_j$ is the homogeneous part of $g$ of degree $j$.
  Suppose that $g \neq \id$, let $j_0$ be the lowest index $j \geq 2$ such that $A_j \neq 0$. We rewrite $g$ as $g =
  \id + A_{j_0} + B$ where $B = \sum_{j > j_0} A_j$ and compute the second iterate

  \begin{align*}
    g^2 (\x) &= g(\x) + A_{j_0} (g(\x)) + B (g(\x)) \\
      &= \id + A_{j_0} (\x) +  B(\x) + A_{j_0} (\x + A_{j_0} (\x) + B(\x) ) + B(g (\x)) \\
      &= \id + 2 A_{j_0} (\x) + (\text{terms of higher degree}).
  \end{align*}

  And for every $k \geq 1$ we obtain
  \[ g^k (\x) = \id + k A_{j_0} (\x) + (\text{terms of higher degree}). \]

  Since, $g$ is of order $p^t$ for a certain $t>0$, replacing $k$ by $p^t$ in this formula we get $ p^t A_{j_0} (\x) =0 $,
  a contradiction since $\Char F \neq p$.
\end{proof}

\begin{rmq}\label{remarkChar0ResteInj}
  If $F$ is of characteristic $0$ and $x_0$ is fixed by $G$, then the proof shows also that $\Phi: g \mapsto
  D_{x_0} g$ is injective.
\end{rmq}

\begin{proof}[Proof of Theorem \ref{SchurBoundPolynomial} when $\k = \Q$]
  As in the linear case, we can find an integer $N$ such that $G \subset \Aut (\A_{\Z[1 / N]}^d)$. So, for $\ell > N$ prime ,
  reduction modulo $\ell$ is well defined on $G$. Now, for $\ell$ large enough such that $\ell$ does not
  divide any coefficient of $g - \id$ for all $g \in G \subset \Aut (\A_{\Z [ 1/N]}^d)$, this homomorphism is injective and we
  can use Dirichlet's theorem to ensure that $\ell$ is a generator of $(\Z / p^2 \Z)^{\times}$. $G$ is now embedded in
  $\Aut (\A_{\F_\ell}^d)$ and we replace it by its image in $\Aut(\A^d_{\F_\ell})$. By Lemma \ref{PointFixeEtDifferentielleInj},
  there is a point $x_0 \in \F_\ell^d$ fixed by $G$ and we have an injective homomorphism $\Phi: G \hookrightarrow \GL_d
  (\F_\ell)$. This concludes the proof when $p \neq 2$.

\end{proof}

\subsection{Extension of Minkoswski's bound to number fields} \label{SubsecSchurBoundNumberFields}

\paragraph{Strategy.--} This part is dedicated to the proof of Schur's bound for finite $p$-groups of
polynomial automorphisms over arbitrary number fields. We will then prove Theorem
\ref{SchurBoundPolynomial} using a Sylow argument. As in the previous section, we want to show the

\begin{thm}\label{theoremPlongementCorpsFiniBonCardinal}
  Let $\k$ be a number field, $d$ an integer and $p$ be an odd prime. Let $G$ be a finite $p$-subgroup of $\Aut_\k
  (\A^d)$, then there exists a finite field $\F$ with $\Char \F \neq p$ and an injective group homomorphism $G \hookrightarrow \GL_d
  (\F)$ such that $v_p (\abs{\GL_d(\F)}) \leq M_\k (d,p)$.
\end{thm}

Indeed, this would prove that $v_p ( \abs G) \leq v_p (\abs{\GL_d (\F)}) \leq M_\k (d,p)$. The natural idea is to do an
analog of the proof for $\k =\Q$. Replace $\Z$ by the ring of integers $L := \mathcal O_\k$ of $\k$, then for any maximal ideal
$\m$ of $L$ lying over a sufficiently large prime, there is an injective homomorphism $G \hookrightarrow \Aut
(\A_{L / \m}^d)$. By taking differentials at a fixed point over $L / \m$ we would see $G$ as a subgroup of $\GL_d( L / \m)$ and
the order of $\GL_d( L /\m)$ would give a bound $v_p (\abs{G}) \leq \sum_{i=1}^d v_p (\abs{L / \m}^i -1)$. The remaining
part is to choose $\m$ wisely so that we get the lowest bound possible. To do this, we use cyclotomic characters.

\paragraph{Cyclotomic characters.--}In this part, $\k$ is a finitely generated field over $\Q$. We denote
by $\mu_{n}$ the
group of $n$-th roots of unity in $\overline \k$.
Recall that $\Aut (\mu_{n}) = (\Z / n \Z)^{\times}$ because every automorphism $\phi$ is of the form $\phi(\omega) =
\omega^a$ where $a \in (\Z / n \Z)^\times$.
\begin{dfn}[Cyclotomic character]
  Denote by $\Gamma_\k = \Gal(\overline \k / \k)$ the absolute Galois
  group of $\k$. For every $n \geq 1, \Gamma_\k$ preserves the group $\mu_{n} \subset \overline \k^{\times}$ of
  $n$-th roots of unity, this induces a group homomorphism
  \[
  \chi_n : \Gamma_\k \rightarrow \Aut (\mu_n) = (\Z / n \Z)^\times  \]
  called the \emph{$n$-th cyclotomic character of $\k$}. In particular, if $p$ is a prime number, since the
  inclusion $\mu_{p^n} \subset \mu_{p^{n+1}}$ induces a group homomorphism $\Aut
  (\mu_{p^{n+1}}) = (\Z / p^{n+1} \Z)^\times \rightarrow \Aut (\mu_{p^n}) = (\Z / p^n \Z)^\times$, we have a compatible
  family of homomorphisms
  \[ \chi_{p^n} : \Gamma_\k \rightarrow \Aut (\mu_{p^n} ). \]
  This family of homomorphisms induces the $p^\infty$-\emph{cyclotomic character}
  \[ \chi_{p^\infty} : \Gamma_\k \rightarrow \Z_p^{\times} = \lim_{\longleftarrow} (\Z / p^n \Z)^{\times} \]
  where $\Z_p$ is the ring of $p$-adic integers. This homomorphism is continuous with respect to the profinite topologies
  on $\Gamma_\k$ and $\Z_p^\times$.
\end{dfn}

We are interested in the image of $\chi_{p^\infty}$ which is a closed subgroup of $\Z_p^{\times}$. Define $t(\k; p)$ and
$m(\k; p)$ as in Section \ref{par:Minkowski_Schur}. The number $m(\k; p)$ is always finite if $\k$ is
finitely generated over $\Q$ (see \cite{SerreBoundsOrder}, $\S4.3$). If $s$ is an integer, we denote by
  $C_s$ the cyclic group of order $s$.

  \begin{prop}[\cite{SerreBoundsOrder}, $\S 4$]\label{ImageCaractereCyclotomique} $\phantom{-}$
  \begin{enumerate}
    \item If $p$ is an odd prime, one has
      \[ \Z_p^{\times} \simeq C_{p-1} \times (1 + p \cdot \Z_p). \]
      The group $1 + p \cdot \Z_p$ is a procyclic subgroup generated by $1+p$ as
      a topological group and isomorphic to the additive group $\Z_p$. Its closed subgroups are the groups $1 + p^j \Z_p$ with
      $j\geq 1$.

      Furthermore, one has
      \[ \im \chi_{p^\infty} = C_{t(\k;p)} \times \left\{  1 + p^{m(\k;p)} \cdot \Z_p \right\}. \]
    \item If $p=2$, then $\Z_2^{\times} = C_2 \times \left\{ 1 + 4 \cdot \Z_2 \right\}$. There are 3
      possibilities for $\im \chi_{2^\infty}$:
      \begin{enumerate}
        \item $\im \chi_{2^\infty} = 1 + 2^{m(\k;p)} \cdot \Z_2$ and then $t(\k;p)=1$.
        \item $\im \chi_{2^\infty} = \langle -1 + 2^{m(\k;p)-1} \rangle$ (the closure of the group generated by $-1 +
          2^{m(\k;p)-1} $) and then $t(\k;p)=2$.
        \item $\im \chi_{2^\infty} = C_2 \times \left\{  1 + 2^{m(\k;p)} \Z_2 \right\}$ and then $t(\k;p)=2$.
      \end{enumerate}
  \end{enumerate}
\end{prop}

\begin{rmq}
  Those 3 cases are distinct when $m(\k,p) \neq \infty$. We will refer as $\k$ being in case (a), (b), or (c) when $\im
  \chi_{2^\infty}$ is of the form (a),(b) or (c) of Proposition \ref{ImageCaractereCyclotomique}.
\end{rmq}

Recall that an integral domain $L$ is \emph{normal} if every localisation at a prime ideal
of $L$ is integrally closed. Let $L$ be a normal domain that is finitely generated over $\Z$ such that the
fraction field of $L$ is $\k$. For
any maximal ideal $\m \subset L$, the quotient $L / \m$ is finite by the Nullstellensatz for Jacobson
rings and $N(\m) := \abs{L / \m}$ is the \emph{norm} of
$\m$. Recall that for a ring $R$, $\Spec R$ denotes the set of prime ideals of $R$ and $\Specmax R$ the set of its
maximal ideals both with the Zariski topology. The following theorem is proven in \cite[$\S 6$ Theorem
7]{SerreBoundsOrder}.

\begin{thm}\label{OuvertDenseAvecAnneauNormal}
  Let $L$ be a normal domain finitely generated over $\Z$ such that the fraction field of $L$ is $\k$. Let $n$
  be an integer and $c$ an element of $(\Z / n\Z)^{\times}$. Denote by $X_c$ the set of elements $x \in \Specmax(L)$ such
  that $N(x) \equiv c \mod n$. Then:
  \begin{enumerate}
    \item If $c \not \in \im \chi_n$, $X_c = \emptyset$.
    \item If $c \in \im \chi_n$, then $X_c$ is Zariski-dense in $\Specmax (L)$. In particular, $X_c$ is infinite.
  \end{enumerate}
\end{thm}

In particular, the ring of integers of a number field is normal because it is integrally
closed and this property is stable under localisation. So Theorem \ref{OuvertDenseAvecAnneauNormal} holds
for $L$ the ring of integers of a number field.

\paragraph{Valuations.--}
  We define the constant
  \[ M'_\k (d,p) = \inf_{u \in \im \chi_{p^\infty}} \sum_{i=1}^d v_p (u^i -1). \]

The next proposition is adapted from Proposition 4, $\S 6$ of \cite{SerreBoundsOrder} to our context.

\begin{prop}\label{ConstanteMajoréeParBorneSchur}
  One has
  \begin{enumerate}[label=(\alph*)]
    \item If $p \neq 2$ or if $p=2$ and $t(\k;p)=1$ ($\k$ is in case (a)), then
      \[ M'_\k (d,p) = \sum_{\substack{ i=1 \\ t(\k;p) | i}}^d (m(\k;p)+ v_p(i)) = M_\k (d,p). \]
    \item If $p=2$, $t(\k;p)=2$ and $\k$ is in case (b), one has
      \[ M'_\k (d,2) = r_1 + (m(\k;p)-1)r_0 + \sum_{i=1}^d v_2 (i) = M_\k (d,2) \]
      where $r_1$ is the number of odd integers between $1$ and $d$ and $r_0$ the number of even integers
      in this range.
    \item If $p=2$, $t(\k;p)=2$ and $\k$ is in case(c), one has
      \[ M'_\k(d,2) = r_1 + m(\k;p)r_0 + \sum_{i=1}^d v_2 (i) = \ent{\frac{d}{2}} + M_\k (d,2) \]
      with the same definition for $r_1$ and $r_0$.
  \end{enumerate}
\end{prop}

\begin{proof}
  Set $t= t(\k;p), m = m(\k;p)$.
  We start with the case $p \neq 2$. First if $t$ divides $i$, then $v_p (u^i -1) \geq m + v_p(i)$. This is because $u$
  can be written as $zv$ with $z^t =1$ and $v_p (v-1) \geq m$, so $v_p (u^i -1) = v_p (v^i-1)$. So we have an inequality
  $M'_\k(d,p) \geq \sum_{\substack{ i=1 \\ t | i}}^d (m+ v_p(i))$. To have the opposite one, choose $u \in \im
  \chi_{p^\infty}$ such that $u = zx$ with $z$ of order $t$ and $v_p(x -1) = m$. This also works for $p=2$ and $t=1$.

  Suppose now that $p=2$ and $t=2$, Define $m'= m-1$ in case (b) and $m' = m$ in case (c). Then for every $x \in \im \chi_{2^\infty}$,
  \begin{align*}
v_2(x^i -1) &\geq m' + v_2 (i) \text{ if } i \text{ is even.} \\
  v_2(x^i -1)& \geq 1 \text{ if } i \text{ is odd.}
  \end{align*}
  This gives
  \[ M'_\k(d,2) \geq \sum_{i \text{ odd}} 1 + \sum_{i \text{ even}} (m' + v_2 (i) ) = r_1 + m' r_0 + \sum_{i \text{
  even}} v_2 (i).\]

  To show the opposite inequality, we use the fact that $x =-1 + 2^{m'} \in \im \chi_{2^\infty}$ and we check that $\sum_{i=1}^d v_2(x^i-1) = r_1 + m'r_0 + \sum_{i=1}^d v_2 (i)$.

  Now, to show the different equalities, notice that for (a):
  \[ M'_\k (d,p) = m \cdot \ent{\frac{d}{t}} + \sum_{i=1}^{\ent{\frac{d}{t}}} v_p(ti). \]
  Now, since $t$ divides $p-1$, one has $v_p(ti) = v_p(i)$ and the rest of the computation is similar as in the case $\k = \Q$.

  For (b) and (c), we have $r_0 = \ent{\frac{d}{2}}$ and $r_1 = d - r_0$.
  \begin{align*}
M'_\k(d,2) &\leq d - \ent{\frac{d}{2}} + m' \ent{\frac{d}{2}} + \sum_{i=1}^d v_2(i)\\
      &= d + (m' -1) \ent{\frac{d}{2}} + \sum_{i=1}^d v_2 (i)  \\
      &= d + (m'-1) \ent{\frac{d}{t}} + \sum_{k \geq 1} \ent{\frac{d}{2^k}} \\
      &= d + m' \ent{\frac{d}{t}} + \sum_{k\geq 1} \ent{\frac{d}{2^k t}}.
  \end{align*}
\end{proof}

We can now state Theorem \ref{theoremPlongementCorpsFiniBonCardinal} without assuming $p$ odd.

\begin{thm}\label{theoremPlongementCorpsFiniBonCardinalPGeneral}
  Let $\k$ be a number field, $d$ an integer and $p$ be prime. Let $G$ be a finite $p$-subgroup of $\Aut_\k
  (\A^d)$, then there exists a finite field $\F$ with $\Char \F \neq p$ and an injective group homomorphism $G \hookrightarrow \GL_d
  (\F)$ such that $v_p (\abs{\GL_d(\F)}) \leq M_\k ' (d,p)$.
\end{thm}

\paragraph{Proof of Theorem \ref{theoremPlongementCorpsFiniBonCardinalPGeneral}.--}
  Take $G$ a finite $p$-subgroup of $\Aut (\A_\k^d)$ with $p$ prime.
  \smallskip

  {\sl{Step 1. Reduction modulo $\mathfrak l $.-- }}
   Set $L = \mathcal O_\k$. For every
   element $a \in \k^\times$ the fractional ideal generated by $a$ is of the form (see \cite{Neukirch},
   $\S 3$)
  \[ a \cdot \mathcal O_\k = (a) = \prod_{\mathfrak l \in \Spec L} \mathfrak l^{v_\mathfrak l (a)} \]
  and the prime ideals $\mathfrak l$ such that $v_\mathfrak l (a) \neq 0$ are in finite number. For such an $\mathfrak l$
  there exists a unique prime $\ell \in \Z_+$ such that $(\ell) \subset \mathfrak l$. We define for $g \in \Aut (\A_\k^d)$
  \[ \ell_g := \max_{a \in \text{coeff}(g - \id)} \left\{ \text{prime } \ell \in \Z_+ :  \exists \mathfrak l \in \Spec L, (\ell) \subset \mathfrak l, v_\mathfrak l (a) \neq 0 \right\} \]
  where $\text{coeff} (g- \id)$ is the set of coefficients of the polynomial transformation $g - \id$. Set $M_1 = \max_{g
  \in G} \ell_g$ ($M_1 < + \infty$ since $G$ is finite) and $M = \max(M_1,p)$, then for every prime $\ell >M$ and for
  every $\m \in \Specmax(L)$ such that $(\ell) \subset \m$, we have a well-defined injective homomorphism
  \[ \Psi: G \hookrightarrow
    \Aut( \A_{\F}^d),
\]
where $\F = L / \m$. Indeed, the homomorphism of rings $\phi: L \twoheadrightarrow L / \m$ induces the homomorphism $\phi: L_\m :=
  \inv{(L \setminus \m)} L \rightarrow L/ \m$. By construction, $G$ is a subgroup of $\Aut (\A_{L_\m}^d)$, so $\phi: G
  \rightarrow \Aut (\A_{L / \m}^d)$ is well-defined and it is injective by our definition of $M$.

  \smallskip
  {\sl{Step 2. The group $\Psi(G)$.--}}
  Now, $\Psi(G)$ is a $p$-subgroup of $\Aut (\A_{\F}^d)$. Since $p \not \in \m$, we get $\Char (\F) \neq p$. By
  Proposition \ref{PointFixeEtDifferentielleInj}, there is a point $x_0$ in $\A^d(\F)$
  fixed by $\Psi(G)$ and by taking the differentials at $x_0$, we obtain an injective homomorphism $G \hookrightarrow \Psi(G) \hookrightarrow \GL_d
  (\F)$.  So, we get
  \begin{equation}\label{eq1}
    v_p (\abs G) \leq v_p \left( N(\m)^{\frac{d(d+1)}{2}} \prod_{i=1}^d (N(\m)^i-1) \right) = \sum_{i=1}^d v_p(N(\m)^i -1).
  \end{equation}

  Set $X := \left\{  \m \in \Specmax (L) :  \m | (s), \text{for some }s > M \text{ prime} \right\}$, then
  (\ref{eq1}) holds for
  all $\m \in X$ and we obtain $v_p (\abs G) \leq \inf_{\m \in X} \sum_{i=1}^d v_p(N(\m)^i -1).$ So, to conclude, all we
  have to prove is
  \begin{equation}
        \inf_{\m \in X} \sum_{i=1}^d v_p(N(\m)^i -1) \leq M'_\k(d,p).
    \label{eqInegalite}
  \end{equation}

  \smallskip
  {\sl{Step 3. Proof of \eqref{eqInegalite}.--}} The set $X$ is open in $\Specmax L$. For, $X = \left( \bigcup_{l \leq
  M, l \text{ prime}} V(l) \right)^c$ with $V(l) = \left\{  \m \in \Specmax(L) :  (l) \subset \m \right\}$
  and $V(l)$ is closed.
   Take $u \in \im \chi_{p^\infty}$. For $j \geq 1$, let $u_j$ be the projection
   of $u$ in $(\Z / p^j \Z)^{\times}$. By Theorem \ref{OuvertDenseAvecAnneauNormal} the set of maximal ideals $\m$
   such that $N(\m) \equiv u_j \mod p^j$ is dense, therefore it intersects the open subset $X$, so for every $j\geq
   1$, we can find $\m_j \in X$ such that $ N(\m_j) \equiv u_j \mod p^j$. Then, one has $\lim_{j \rightarrow
   \infty} N( \m_j) = u$ in $\Z_p^{\times}$, therefore $v_p (u^i -1) = \lim_{j \rightarrow \infty} v_p( N(\m_j)^i
   -1)$ so
   \[ \inf_{\m \in X} \sum_{i=1}^d v_p(N(\m)^i -1) \leq \sum_{i=1}^d v_p(u^i-1); \]
   and this holds for every $u \in \im \chi_{p^\infty}$. Using Proposition \ref{ConstanteMajoréeParBorneSchur}, we get
   \[ \inf_{\m \in X} \sum_{i=1}^d v_p(N(\m)^i -1) \leq \inf_{u \in \im_{\chi_{p^\infty}}} \sum_{i=1}^d v_p(u^i-1)
   = M'_\k( d,p). \]

   \paragraph{Proof of Theorem \ref{SchurBoundPolynomial} and comments.--}

   \begin{bigtheorem}\label{BigThmSchurBoundPolynomial}
     Let $\k$ be a number field, let $d$ be a natural number, and let $p$ be a prime. Let $G$ be a finite $p$-subgroup of $\Aut( \A^d_\k)$, then
  \begin{enumerate}
    \item If $p \geq 3$ or $p=2$ and $\k$ is in case (a) or (b), there exists a group embedding
      \[ G \hookrightarrow \GL_d (\k). \]
    \item If $p=2$ and $\k$ is in case (c), there exists a group embedding
      \[ G \hookrightarrow \GL_d (\k(z_4)). \]
  \end{enumerate}
\end{bigtheorem}

\begin{rmq}
  We do not state a Sylow-like property, saying that $G$ is conjugated to a subgroup of $\GL_d (\k)$, we
  only state that we can find an isomorphism of abstract groups from $G$ to a subgroup of $\GL_d (\k)$.
\end{rmq}

\begin{proof}
  For 1, we know that $v_p(\abs G) \leq M_\k (d,p)$ and that there exists a subgroup $H \subset \GL_d
  (\k)$ such that $\abs H = p^{M_\k (d,p)}$ by Theorem \ref{SchurBound}. Let $L = \mathcal O_\k$ be the ring of
  integers of $\k$. The proof of Theorem \ref{theoremPlongementCorpsFiniBonCardinalPGeneral} shows that there
  exists an infinite number of maximal
  ideals $\m$ of $L$ such that $v_p (\GL_d (\F)) \leq M_\k (d,p)$ where $\F = L / \m$. So for any such
  maximal ideal $\m \subset L$ lying over a sufficiently large prime, there are embeddings $\Psi_H: H
  \hookrightarrow \GL_d (\F)$ and $\Psi_G : G \hookrightarrow \GL_d (\F)$. Looking at the size of $H$, we
  deduce that $v_p (\GL_d (\F)) = M_\k (d,p)$ and $\Psi_H (H)$ is a $p$-Sylow of $\GL_d(\F)$. By Sylow's
  theorems, $\Psi_G(G)$ is conjugated to a subgroup of $\Psi_H(H)$ in $\GL_d(\F)$. This implies that $G$
  is isomorphic to a subgroup of $H$.

  For 2, if $\k$ is in case (c) then one can check that $\k(z_4)$ is in case (a) and that $m(\k(z_4); 2) =
  m(\k; 2)$, therefore $M_{\k(z_4)}(d,2) = M_\k (d,2) + \ent{\frac{d}{2}}$ and the same proof as 1 shows
  the result.
\end{proof}

\begin{rmq}\label{RmqFinitelyGeneratedField}
  Theorem \ref{SchurBoundPolynomial} and \ref{BigThmSchurBoundPolynomial} still hold for $\k$ finitely
  generated over $\Q$. We just need to explain how the proof of Theorem
  \ref{theoremPlongementCorpsFiniBonCardinalPGeneral} works in that case.

  We need to find a normal domain $L$ finitely generated over $\Z$ such that $G$ is defined over $L$ and
  to define the open subset $X \subset \Specmax L$ used for equation \eqref{eqInegalite}. Here
  is how to proceed: since $G$ is finite, there exists a
  finitely generated $\Z$-algebra $R$ such that the elements of $G$ are defined over $R$, we can suppose that $R$
  contains $1/p$. By Noether Normalization's Lemma and more precisely by generic freeness (see
  \cite{eisenbud1995commutative}, Theorem 14.4), there exists $t_1, \dots, t_s \in R$ and an integer $N$ such that
  $R$ is a finite free module over $\Z[1/N][t_1,\dots, t_s]$. We can then take for $L$ the integral closure of
  $\Z[1/N][t_1, \dots, t_s]$ in $\k$, $L$ is a normal domain over which
  $G$ is defined since $R \subset L$. We also have that $L$ is finitely generated over $\Z$ because by
  \cite[Theorem 4.14]{eisenbud1995commutative} it is a finite module over $\Z[1/N] [t_1, \dots, t_s]$.

  Now, let $A$ be the set of coefficients of $g - \id$ for $g \in G$. Set $X = \left\{ \m \in \Specmax L :
  A \cap \m = \emptyset \right\}$. This is an open subset of $\Specmax L$ as $A$ is finite and $X = \bigcap_{a \in A}
  V(a)^c$. For any $\m \in X$ we have an injective group homomorphism $G \hookrightarrow \Aut (\A^d_{L
  / \m})$ and Equation \eqref{eq1} holds. The proof of Equation \eqref{eqInegalite} is the same as in the
  case of number fields. This proves Theorem \ref{theoremPlongementCorpsFiniBonCardinalPGeneral} for finitely generated fields over
  $\Q$.
\end{rmq}

   To prove Theorem \ref{BigThmSchurBoundPolynomial}, the key ingredient is
   that there exists subgroups of $\GL_d(\k)$ of size $p^{M_\k(d,p)}$, as Theorem
   \ref{SchurBoundPolynomial} is stated only for number fields we show for completeness how to construct finite
   $p$-groups of $\GL_d (\k)$ of size $p^{M_\k(d,p)}$ when $\k$ is finitely generated over $\Q$. The proof
   of Theorem \ref{BigThmSchurBoundPolynomial} for finitely generated fields over $\Q$ is then similar as
   in the case of number fields using Noether Normalization Lemma, we leave the details to the reader.

   \begin{prop}
     Let $\k$ be a finitely generated field over $\Q$ and let $p$ be a prime, there exists a finite
     $p$-subgroup of $\GL_d(\k)$ of size $p^{M_\k(d,p)}$.
     \label{PropBorneOptimale}
   \end{prop}

   \begin{proof} Set $t= t(\k; p), m = m(\k;p)$ and $r = \lfloor d/t \rfloor$.

     \paragraph{The case $p \geq 3$.--} Let $\rho = z_{p^m} \in
         \k(z_p)$. Then, the group $(\Z / p^m \Z)$ acts on $\k(z_p)$ via multiplication by $\rho^k$ for all $k \in \Z / p^m \Z$.
         Now take $r$ copies of $\k(z_p)$; this is a $\k$-vector space $V$ of dimension $t \cdot r \leq d$ and let $S_r$ be the
         $r$-th symmetric group, $S_r$ acts on $V$ by permuting the $r$ copies of $\k(z_p)$ and therefore the group
         \[ G := S_r \ltimes (\Z / p^m \Z)^r \]
         acts faithfully by linear automorphisms on $V$ and has the desired size. Indeed, $v_p (\vert G \vert) =
         m \cdot \left\lfloor \frac{d}{t} \right\rfloor + v_p (\lfloor \frac{d}{t} \rfloor !)$.

     \paragraph{The case $p=2$ and $t=1$.--} In that case, $\k = \k (z_4)$, then $M_\k(d,2) = m \cdot \left\lfloor
       \frac{d}{t} \right\rfloor + v_2 ( \left\lfloor \frac{d}{t} \right\rfloor !)$. Therefore, the proof above works
       as well, with $\rho = z_{2^m}$ acting on $\k(z_4) = \k$.

       \paragraph{The case $p=2$ and $t=2$.--} The construction above
         yields that $(\Z / 2^m \Z)$ acts linearly on $\k(z_4)$. We twist this action by the Galois
         automorphism $\sigma$ that sends $z_4$ to $-z_4$; $\sigma$ is an involution that sends $\rho = z_{2^m}$ to another
         primitive $2^m$-th root of unity. So we get that the
         group $H:= \Z / 2 \Z \ltimes \Z / 2^m \Z$ acts faithfully on $\k(z_4)$. Now set $ r = \lfloor d/2
         \rfloor$, we have that $G := S_r \ltimes H$ acts faithfully and linearly on a $\k$ vector space
         $V$ consisting of $r$ copies of $\k(z_4)$. The vector space $V$ has dimension $2 \cdot \lfloor d/2
         \rfloor \leq d$. Now, we have
         \[ v_2 (\vert G \vert) = (m+1) \cdot \lfloor d/2 \rfloor + v_2 (\lfloor d/2 \rfloor !). \]
         If $d$ is even this is equal to $M_\k (d,2)$ and we are done. If $d$ is odd then $v_2 (\abs G) = M_\k (d,2) -1$ but
         then $V$ is of dimension $d-1$ so the group $G \times \{ \pm 1 \}$ acts faithfully on $V \oplus \k$
         that is of dimension $d$ and this group has the desired size.

   \end{proof}

   We can therefore state:

   \begin{bigtheorem}\label{BigThMSchurBoundPolynomialGeneralCase}
     Let $\k$ be a finitely generated field over $\Q$, let $d$ be a natural number, and let $p$ be a
     prime. Let $G$ be a finite $p$-subgroup of $\Aut( \A^d_\k)$, then
  \begin{enumerate}
    \item If $p \geq 3$ or $p=2$ and $\k$ is in case (a) or (b), there exists a group embedding
      $ G \hookrightarrow \GL_d (\k)$   and $ v_p (\vert G \vert) \leq M_\k (d;p)$.

    \item If $p=2$ and $\k$ is in case (c), there exists a group embedding
      $ G \hookrightarrow \GL_d (\k(z_4))$ and $v_2(\vert G \vert) \leq M_\k(d, 2) + \lfloor \frac{d}{2} \rfloor $.
  \end{enumerate}
\end{bigtheorem}

\begin{rmq}\label{remarkCasPairNeMarchPas}
  We get the optimal bounds except when $p=2$ and $\k$ is in case (c) (this includes $\k = \Q$). For that case,
  following Remark \ref{remarkCasPairLinear}, to get the optimal bound one would need a result of the following type:
  \emph{
    Let $\k$ be a number field in case (c) and $G$ a finite subgroup of $\Aut(\A_\k^d)$ of order $2^\alpha$, then for $\m$ in the complement of
  a finite set of $\Specmax \mathcal O_\k$ the group $G$ embeds into an orthogonal group over $\mathcal O_\k / \m$.}

  We know that for any maximal ideal $\m$ lying over a large
  enough prime, there exists an embedding $G \hookrightarrow \GL_d(\F)$ and a fixed point $\bar x \in (\F)^d$ of
  $G$ where $\F = \mathcal O_\k / \m$. The problem is to find a symmetric matrix $A$ such that
  \[
    A_G := \sum_{g \in G} {}^t \! D_{\bar x} g \cdot A \cdot D_{\bar x} g
  \]
  is non-degenerate. Such an $A$ does not exist for every subgroup of $\GL_d (\F)$ precisely because $v_2
  (\abs{\GL_d(\F})$ is larger than the 2-adic valuation of the order of any orthogonal group over $\F$. So we have to
    use that $G$ comes from a group over $\k$ and adapt $\m$ wisely.

    Here is one way to attack this problem. Pick a fixed point $\overline x$ of $G$ with coordinates in $\overline \Q$;
    such a point exist because otherwise let $(P_n)$ be the system of polynomial equations stating that $G$ has a fixed
    point. If this system has no solution over $\overline \Q$ then by Hilbert's Nullstellensatz, there is a relation of
    the form $1 = \sum Q_i P_i$ for some polynomials $Q_i$. Now take a number field $\k'$ where this relation is defined.
    By the previous paragraph we can reduce modulo a large enough maximal ideal $\m$ of $\mathcal O_{\k'}$ (i.e
    lying over a large enough prime) and this would yield an
    injective group homomorphism $G \hookrightarrow \Aut (\A^d)$ where $F$ is a finite field with $\Char F \neq p$. The relation
    $1 = \sum Q_i P_i $ still holds in $\F$ but this is absurd since we know that $G$ admits a fixed point over $\F$.
     Let $\k '$ be the number field generated by the coordinates of $\overline x$
    and $\k$. We would like to find $A$ such that $A_G$ is non-degenerate. If $\k ' \subset \R$ we can use argument of
    positive definiteness to do so, but otherwise a first difficulty occurs. Now, even if such an $A$ could be found,
    the arithmetic of $\k '$ leads to another difficulty: For any maximal ideal $\m ' \subset \mathcal O_{\k '}$ lying over a
    large enough maximal ideal $\m \subset \mathcal O_\k$, the image $x'$ of $\overline x$ in $\F ' = \mathcal O_{\k'} / \m'$ is a fixed
    point of $G$, and the reduction modulo $\m'$ of $A_G$ is an invertible symmetric matrix over $\F'$. But if the
    degree $[ \F', \F]$ is even, then the 2-adic valuation of any orthogonal group over $\F '$ will be too large to get
    the optimal bound.
\end{rmq}

\section{$p$-adic analysis}\label{SecPAdicAnalysis}
To prove Theorem \ref{BoundNilpotentGroups}, we will show that any
finitely generated nilpotent group acting on a complex quasiprojective variety of dimension~$d$~can be embedded in a
finite dimensional~$p$-adic Lie group acting analytically on a~$p$-adic manifold of dimension~$d$. The
theorem will follow from a version of Theorem 1.1 of \cite{epstein1979transformation} in a~$p$-adic
context. In this section, we introduce all the tools from~$p$-adic analysis and~$p$-adic Lie groups needed
for the proof.

\subsection{Tate-Analytic Diffeomorphisms}\label{SecAnalyticDiffeo}

\subsubsection{Definitions and topology}

Let~$p$~be a prime. We denote by~$\Z_p$~the completed ring of~$\Z$~with respect to the~$p$-adic norm defined such
that~$\abs p = 1/p$. Denote by~$\Q_p$~the completion of~$\Q$~with respect to this norm. Then~$\Q_p =
\Frac(\Z_p)$~and~$\Z_p$~is the set of elements of~$\Q_p$~of absolute value~$\leq 1$. We extend this norm
to~$\Q_p^d$~by taking the maximum of the absolute values of the coordinates. We will use explicitly the ring~$\Z_p$~and
the field~$\Q_p$~but what follows can be done with any complete valued ring or field of characteristic~$0$. The right
setup would be to consider~$\C_p$~the completion of the algebraic closure of~$\Q_p$~and~$\D_p$~the unit ball of
$\C_p$.

For
reference, check \cite{cantat2014algebraic}.
We denote by~$B(x,r) = \left\{ y \in \Q_p^d : \norm{x-y} \leq r \right\}$~the closed ball of radius~$r$~and center
$x$. It is both open and closed. Such sets will be called \emph{clopen}.

\paragraph{Tate analytic maps.--}Classically, a function~$\Z_p^d
\rightarrow \Q_p$~is analytic if it can be written locally as a converging power series, we work with
\emph{Tate-analytic} functions which are converging power series of radius~$\geq 1$~over~$\Z_p^d$.

Take~$\Z_p^d$~with its standard coordinates~$\x = x_1, \cdots, x_d$.
On~$\Q_p[x_1,\cdots,x_d] =: \Q_p [\x]$~the Gauss norm is defined by
\[ \forall g \in \Q_p[\x], \quad g = \sum_{I \subset \Z_+^d} a_I
\x^I, \quad \norm g := \max_{I} \abs{a_I} \]
where $I = (I_1, \cdots, I_d)$ and $\x^I := x_1^{I_1}\cdots x_d^{I_d}$;
we denote by~$\Q_p \langle x_1,\cdots, x_d \rangle =: \Q_p \langle \x
\rangle$~the completion of $\Q_p [x_1,\cdots,x_d]$ with respect to the Gauss norm~$\Q_p \langle \x \rangle$~is the set of
formal power series with coefficients in~$\Q_p$~such that~$a_I \rightarrow 0$~when~$I \rightarrow \infty$~(i.e when
$\max(I) \rightarrow \infty$). It is also the set
of formal power series with coefficients in~$\Q_p$~converging over~$\Z_p^d$. This shows that~$\Q_p \langle \x \rangle$
equipped with the Gauss norm is an infinite-dimensional Banach space over~$\Q_p$. For all polynomials~$f,g \in
\Q_p [\x]$, then~$\norm{f \cdot g} \leq \norm f \cdot \norm g$~and this is also true in~$\Q_p \langle \x \rangle$,
therefore~$\Q_p \langle \x \rangle$~is a Banach algebra over~$\Q_p$, it is the \emph{Tate algebra} over~$\Q_p$~in~$d$~variables
(see \cite{robert2013course}). We also define~$\Z_p \langle
\x \rangle$~which is the completion of~$\Z_p [\x]$~for the gauss norm; it is in fact the set of elements of~$\Q_p
\langle \x \rangle$~of norm~$\leq 1$.

\begin{rmq}\label{remarkCoeffEntierAMultiplicationPres}
  For each~$f \in \Q_p \langle \x \rangle$~there exists an
  element~$s \in \Z_p$~such that~$s \cdot f \in \Z_p \langle \x \rangle$~and if~$g \in \Q_p \langle \x \rangle$~is such
  that~$g(0) \in \Z_p$, then there exist an integer~$N>0$~such that~$g (p^N \x) \in \Z_p \langle \x \rangle$. Moreover,
  if~$g \in \Q_p [ [ \x ] ]$~is a formal power series with coefficients in~$\Q_p$~with a strictly positive
  radius of convergence, then there exists an integer~$N$~such that~$g (p^N \x)$~belongs to~$\Q_p \langle \x \rangle$.
\end{rmq}

\begin{rmq}\label{remarkPourquoiCoeffEntiers}
  There exist Tate-analytic maps with non-integer coefficients such that~$f(\Z_p^d) \subset \Z_p$. For example, take
  \[ f(x) = \frac{x^p - x}{p}. \]
  Since for all~$x \in \Z_p, x^p \equiv x \mod p$,~$f$~induces a map~$f: \Z_p \rightarrow \Z_p$. However
every element~$f \in \Q_p \langle \x \rangle^d$~induces a map~$f: \D_p^d \rightarrow \C_p$~and we have~$f(\D_p^d) \subset
\D_p \Leftrightarrow f \in \Z_p \langle \x \rangle^d$. This has to do with the residue field of~$\Z_p$~being finite but
not the residue field of~$\D_p$~(see \cite{robert2013course}, Proposition of page 240).
\end{rmq}

For any~$m \geq 0$, elements of~$\Q_p \langle \x \rangle^m$~are called \emph{Tate-analytic functions}.
If~$g \in \Q_p \langle \x \rangle^d$, then
\begin{equation}
  \forall x, y \in \Z_p^d, \norm{g(x) - g(y)} \leq \norm g \norm{x-y}.
  \label{EqLipschitz}
\end{equation}
In particular,~$g$~is~$\norm g$-Lipschitz.

\begin{prop}[Strassman's Theorem, see \cite{robert2013course}, chapter 6, section 2.1] \label{PropIsolatedZeroPrinciple}
  Let~$f \in \Q_p \langle t \rangle$~be a Tate-analytic function in one variable, if~$f$~is
  not the zero function, then~$f$~has a finite number of zeros over~$\Z_p$.
\end{prop}

\begin{cor}\label{AnalyticContinuation}
  Let~$f \in \Q_p \langle \x \rangle$, if there exists a non-empty open subset
 ~$\mathcal U \subset \Z_p^d$~such that~$f_{| \mathcal U} \equiv 0$~then~$f$~is the zero function.
\end{cor}

\begin{rmq}
  This is not true for analytic functions over~$\Z_p^d$. For example define~$g$~by~$g(y) = 1$~if~$\norm y \leq
  \abs p$~and~$g(y) = 0$~otherwise. Then,~$g$~is analytic at every point of~$\Z_p^d$~because it is locally constant, it
  vanishes on the open subset~$\left\{ x \in \Z_p^d : \norm x = 1 \right\}$~but~$g$~is not the zero function.
\end{rmq}

\begin{proof}[Proof of Corollary \ref{AnalyticContinuation}]
  Take~$y \in \mathcal U$~and~$x \in \Z_p^d$. Let~$\varphi$~be the
  function~$\varphi: t \in \Z_p \mapsto f(tx +
  (1-t)y)$. Then~$\varphi$~belongs to~$\Q_p \langle t \rangle$~and it vanishes for any sufficiently small~$t$. By
  Proposition \ref{PropIsolatedZeroPrinciple}, we have that~$\varphi$~is the zero function, therefore~$f(x) = 0$.
\end{proof}

Let~$f,g \in \Q_p \langle \x \rangle$~and~$c>0$, we write~$f \equiv g \mod p^c$~if $\norm{f -g} \leq \abs p^c$ and we
extend such notation componentwise for~$\Q_p \langle \x \rangle^m$~for every~$m \geq 1$.

\begin{ex} \label{ExampleCongruence}
  If~$c=1$~and~$f,g \in \Z_p \langle \x \rangle$, then~$f = \sum_I a_I \x^I
  \equiv \id(\x) \mod p$~means that~$\overline f := \sum_{I} \overline{a_I} \x^I
  = \id (\x)$~where~$\overline{a_I} = a_I \mod p$~is the reduction of~$a_i$~mod~$p \Z_p$.
\end{ex}

\paragraph{Tate analytic diffeomorphisms.--}

The composition determines a natural map
\[
  \begin{array}{cclll}
\Z_p \langle X_1,\cdots,X_n \rangle^m & \times & \Z_p \langle Y_1,\cdots, Y_s \rangle^n & \longrightarrow &
    \Z_p \langle Y_1,\cdots, Y_s \rangle^m \\
    (g_1,\cdots.,g_m) & &(h_1,\cdots,h_n)  & \longmapsto & (g_1(h_1,\cdots,h_n),\cdots, g_m(h_1,\cdots,h_n))
\end{array}
\]
If the three integers~$n,m,s$~are equal to the same integer~$d$,~$(\Z_p \langle \x \rangle^d, \circ)$~becomes a semigroup.  The
invertible elements of this semigroup are called \emph{Tate-analytic diffeomorphisms} and form a group denoted by
$\Diff^{an} (\Z_p^d)$. Using Equation \eqref{EqLipschitz}, we have that~$\Diff^{an}(\Z_p^d)$~acts by isometries on
$\Z_p^d$.

\begin{rmq}
  Following Remark \ref{remarkPourquoiCoeffEntiers}, we see that~$\Diff^{an} (\Z_p^d)$~consists exactly of the elements of
 ~$f \in \Q_p \langle \x \rangle$~that induces a Tate-analytic diffeomorphisms~$f: \D_p^d \rightarrow \D_p^d$.
\end{rmq}

The next proposition shows an easy way to construct Tate-analytic
diffeomorphisms of small polydisks.

\begin{prop}[Local inversion theorem, see \cite{SerreLieGroupsLieAlgebras}]\label{ExistenceInverse}
  Let~$\Phi \in \Z_p [[X_1,\cdots.,X_d]]^d$~be a
  power series with a strictly positive radius of convergence. Suppose that~$\Phi(0) = 0$~and~$\det (D_0 \Phi) \neq
  0$, then there exists a unique~$\Psi \in \Q_p [[X_1,\cdots.,X_d]]^d$, with a strictly positive radius of convergence, such
  that~$\Psi(0) = 0$~and \[ \Phi \circ \Psi (\x) = \Psi \circ \Phi(\x) = \x. \]

  Furthermore,~$\norm{\Psi_n} \leq \max (1, \norm{ \inv {D_0 \Phi}}^n)$, where~$\Psi_n \in \Q_p
  [X_1,\cdots,X_n]^d$~is the homogeneous part of degree~$n$~of~$\Psi$~and~$\abs{\abs{\cdot}}$~is the Gauss norm over
  polynomials. Therefore, if~$\Phi$~belongs to~$\Z_p \langle \x \rangle^d$, then for any~$k$~such that~$~\abs p^k <
  \norm{D_0 \inv \Phi}$, we have that~$\frac{1}{p^k}  \Phi  (p^k \x)$~and~$\frac{1}{p^k}  \Psi  (p^k
  \x)$~are Tate-analytic diffeomorphisms and are inverse of each other.
\end{prop}

\paragraph{Group topology.--}
The following proposition shows that~$\Diff^{an} (\Z_p^d)$~is a topological group with respect to the topology
induced by the Gauss norm.

\begin{prop}\label{truc1}
  Let~$f,g,h \in \Z_p \langle \x \rangle^d$, then
  \begin{enumerate}
    \item~$\norm{g \circ f} \leq \norm g$.
    \item If~$f$~is an element of~$\Diff^{an} (\Z_p^d)$~then~$\norm{ g \circ f} = \norm g$.
    \item~$~\norm{ g \circ (\id +h) - g } \leq \norm h$.
    \item~$\norm{\inv f - \id} = \norm{f - \id}$~if~$f$~is a Tate-analytic diffeomorphism.
  \end{enumerate}
\end{prop}

\begin{lemme}\label{lemma:PuissanceCongruence}
  Let~$f$~be an element of~$\Diff^{an}(\Z_p^d)$, if~$f \equiv \id \mod p$~then~$f^{p^c} \equiv \id \mod p^c$.
\end{lemme}

\begin{cor}\label{corollary:SubgroupsBasisOfNeighbourhoods}
  Let~$c>0$~be a real number, then the
  subgroup~$\Diff^{an}_c(\Z_p^d)$~of~$\Diff^{an} (\Z_p^d)$~consisting of all
  elements~$f \in \Diff^{an} (\Z_p^d)$~such that~$f \equiv \id \mod p^c$~is a normal subgroup of
  ~$\Diff^{an}(\Z_p^d)$.
\end{cor}

Proposition \ref{truc1}, Lemma \ref{lemma:PuissanceCongruence} and Corollary \ref{corollary:SubgroupsBasisOfNeighbourhoods} are
proven in \cite{cantat2014algebraic}, section 2.1.

\subsubsection{Analytic flow and Bell-Poonen theorem}

\paragraph{Flows and vector fields.--}
As in real or complex geometry, we define vector fields and flows. Let~$d$~be an integer:

   A \emph{Tate-analytic vector field}~$\X$~over~$\Z_p^d$~is a vector field of the form
    \[ \X(\x) = \sum_{i=1}^d u_i (\x) \partial_i \]
    where each~$u_i$~belongs to~$\Q_p \langle \x \rangle$. The Lie bracket of two vector
    fields~$\X$~and~$\mathbf Y= \sum_{i=1}^d v_i \partial_i$~is the vector field defined by
    \[ [ \X, \mathbf Y] =
      \sum_{j=1}^d w_j (\x) \partial_j \text{ with } w_j = \sum_{i=1}^d \left(u_i \frac{\partial v_j}{\partial x_i} - v_i
      \frac{\partial u_j}{\partial x_i}\right).
    \]
    The~$\Q_p$-Lie algebra of Tate-analytic vector fields over~$\Z_p^d$~is denoted by~$\Theta(\Z_p^d)$~it is a strict
    subalgebra of the Lie Algebra of analytic vector fields over~$\Z_p^d$. The Gauss norm of a Tate-analytic vector field
   ~$\X = \sum u_i (\x) \partial_i$~is defined as~$\norm \X = \max_i \norm {u_i}$~and makes~$\Theta(\Z_p^d)$~a complete Lie
    Algebra over~$\Q_p$~isomorphic as a Banach space to~$\Q_p \langle \x \rangle^d$.

 A \emph{Tate-analytic flow}~$\Phi$~over~$\Z_p^d$~is an element of~$\Z_p \langle X_1, \cdots,
    X_d, t \rangle^d = \Z_p \langle \x,t \rangle^d$~which satisfies the following properties

    \begin{enumerate}[label=(\roman*)]
      \item~$~\forall \x \in \Z_p^d, \ \forall s,t \in \Z_p, \quad \Phi(\x, s+t) = \Phi( \Phi(\x,s), t).$
      \item~$\forall \x \in \Z_p^d,\quad \Phi(\x,0) = \id(\x)$.
    \end{enumerate}

Set~$\Phi_t := \Phi( \cdot, t) \in \Z_p \langle \x \rangle$. Then,~$\Phi_0 = \id$~and~$\Phi_t \in \Diff^{an}(\Z_p^d)$
since~$\inv{\Phi_t} = \Phi_{-t}$.  Then,~$t \in \Z_p \mapsto \Phi_t \in \Diff^{an}(\Z_p^d)$~is a continuous
homomorphism of topological groups with respect to the Gauss norm. The main point here is that flows are
parametrized by the compact group~$(\Z_p, +)$.

\begin{ex}
  If~$\Phi$~is a Tate-analytic flow, then we can define its associated Tate-analytic vector
  field~$\X_\Phi := \frac{\partial \Phi_t}{\partial t}_{|t=0}$. In particular,~$\X_\Phi$~is~$\Phi_t$-invariant, for all~$t
  \in \Z_p$.
\end{ex}

\paragraph{From vector fields to Tate-analytic flows.--}

Since a Tate-analytic vector field~$\X$~is analytic, it is a general fact that it admits local analytic flows over
$\Z_p^d$~(see \cite{bourbaki2007varietes} for example), the next proposition shows that if the norm
of~$\X$~is sufficiently small, then it
admits a global Tate-analytic flow.

\begin{prop}\label{PropExistenceGlobalTateAnalyticFlow}
  If~$\X$~is a Tate-analytic flow over~$\Z_p^d$, then for any sufficiently small~$\lambda \in \Z_p$, there
  exists a unique Tate-analytic flow~$\Phi^\lambda \in \Z_p \langle \x, t \rangle^d$~such that
  \[ \frac{\partial \Phi_t^\lambda (\x)}{\partial t} = \lambda \X(\Phi_t^\lambda(\x)). \]

  In particular, let $c>0$ be such that $c > \frac{1}{p-1}$, then every Tate-analytic vector
  fields~$\X$~such that~$\norm \X \leq \abs p^c$~admits a global Tate-analytic flow.
\end{prop}

\begin{proof}
  The strategy is to solve this differential equation in the space of power series~$\Q_p \left[ \left[ \x, t \right]
  \right]^d$~and then to show some properties on the radius of convergence of the solution. We first replace~$\X$~by
 ~$\mu \X$~for some~$\mu \in \Z_p$~such that~$\norm \X \leq 1$. Write~$\X (\x) = \sum_i u_i
  (\x) \partial_i$~with~$u_i \in \Z_p \langle \x \rangle$. We look at the differential equations

  \begin{equation}
    \frac{\partial}{\partial t } f_i (\x, t) = u_i(f(\x,t))
    \label{EqDiff}
  \end{equation}
  with~$f_i \in \Q_p \left[ \left[ \x, t \right] \right]$~and~$f = (f_1, \cdots, f_d)$~such that~$f(\x, 0) = \x$.
  Write
  \[
    f_i (\x ,t) = \sum_{k \geq 0 } a_k^{(i)} (\x) t^k , \quad a_k^{(i)} \in \Q_p \left[ \left[ \x \right] \right]
  \]
  then, the unique solution of this equation is formally given by the formulas~$a_k^{(i)} (\x) = \frac{1}{k !}
  \frac{\partial^k f_i}{\partial t^k} (\x, 0)$. We show that for all integer~$k \geq 0, \frac{\partial^k f_i}{\partial t^k}
  (\x, 0)$~belongs to~$\Z_p
  \langle \x \rangle$~by induction on~$k$. We get~$a_0^{(i)} = x_i$~since~$f(\x, 0) = \id(\x)$~and~$a_1^{(i)} (\x) = u_i(\x)$~by
  Equation \eqref{EqDiff}. Take~$k \geq 2$~and suppose the result to be true for all~$l < k$. By
  differentiating both sides of Equation \eqref{EqDiff}~$k-1$~times with respect to~$t$~and taking~$t=0$, we see that~$\frac{\partial^k
  f_i}{\partial t^k} (\x, 0)$~is obtained by sum and compositions of differentials of orders~$\leq k - 1$~of the
  Tate-analytic function~$u_i \in \Z_p \langle \x \rangle$~and the Tate-analytic functions~$\frac{\partial^l}{\partial
  t^l} f_i (\x, 0) \in \Z_p \langle \x \rangle$~with~$l < k$.
  So~$\frac{\partial^k f_i}{\partial t^k} (\x, 0)$~belongs to~$\Z_p \langle \x \rangle$~by induction.

  The solution~$f$~is then of the form
  \[ f(\x ,t) = \id(\x) + \sum_{k \geq 1} \frac{\partial^k f}{\partial t^k} (\x, 0) \frac{t^k}{k !}. \]

  Now take $\lambda \in \Z_p$, such that $\abs \lambda \leq \abs p ^c$. We have that for all $k \geq 0,
  \frac{\lambda^k}{k!} \in \Z_p$ and $\lambda^k / k!  \rightarrow 0$ in $\Z_p$ when $k \rightarrow \infty$. Then,
  $\Phi^\lambda_t := f( \cdot, \lambda t)$~is a Tate-analytic flow such that~$\frac{\partial
  \Phi_\lambda^t}{\partial t} (\x) = \lambda \X(\Phi_t^\lambda (\x))$.

  For the final statement, take~$\X$~a Tate-analytic vector field such that~$\norm \X \leq \abs p^c$
  and let $s \in \Z_p$ be such that $\abs s = \norm \X$, then $\Y := \frac{1}{s} \X$ has norm $\leq 1$.
  The proof shows that there exists a unique Tate-analytic flow~$\Phi$~such that~$\frac{\partial \Phi_t}{\partial t}_{|t=0} = s \Y = \X$.

\end{proof}

\begin{thm}[local linearisation of vector fields]\label{pAdicFrob}
  Let~$\X_1,\cdots,\X_k$~be Tate-analytic vector fields over ~$\Z_p^d$~such that~$[\X_i, \X_j] = 0$~for all~$1 \leq i,j \leq
  k$. Suppose that there exists a point~$m \in \Z_p^d$~such that the vectors~$\X_i (m)$~are linearly independent. Then,
  there exists a clopen subset~$\mathcal V \subset \Z_p^d$~containing~$m$~and an analytic
  diffeomorphism~$\varphi$~from~$\Z_p^d$~onto~$\mathcal V$~such that~$\varphi^* (X_{i|\mathcal V}) = \partial_i$~and such that
 ~$\varphi^*$~yields an injective Lie Algebra homomorphism~$\Theta(\Z_p^d)_{|\mathcal V}
  \hookrightarrow \Theta(\mathcal V)$.
\end{thm}

\begin{rmq}
  This theorem is well known in~$p$-adic differential geometry with analytic regularity (see
  \cite{bourbaki2007varietes}),
  what is important here is that when changing coordinates we keep the Tate-analytic regularity for vector fields.
\end{rmq}

\begin{proof}
  By translation, we can suppose that~$m=0$. We pick $Y_0 \subset T_0 \Z_p^d$ such that we have the decomposition $T_0 \Z_p^d = \Vect
  ( X_1(0),\cdots, X_k (0)) \oplus Y_0$. Let~$e_{1}, \cdots, e_{d-k}$~be a basis of~$Y_0$. Pick local (analytic)
  coordinates~$(x_1,\cdots,x_k, y_1,\cdots, y_{d-k})$~such that for all~$1 \leq j \leq d-k, \frac{\partial }{\partial y_j} (0) = e_j$.

  Define :~$f: \Z_p^{d-k} \rightarrow \Z_p^d$~by \[ f(y_1,\cdots,y_{d-k}) =(0,\cdots,0, y_1,\cdots,y_{d-k}). \] Take the
  local analytic flows~$\varphi^1,\cdots, \varphi^k$~associated to~$\X_1,\cdots, \X_k$~at~$0$~(here we do not suppose these flows to
  be Tate-analytic)~and consider

  \[
    \begin{array}{crcl}
g&: \Z_p^k \times \Z_p^{d-k} & \longrightarrow & {\Z_p^d}\\
& {(t_1,\cdots,t_k; y)} & \longmapsto& { \varphi^1_{t_1} \circ \cdots \circ \varphi^k_{t_k} (f(y)).}
  \end{array}
\]
     The function~$g$~belongs to~$\Z_p
  \left[ \left[ t_1, \cdots, t_k, \y \right] \right]^d$~with a radius of convergence~$r_g >0$, satisfies~$g(0) = 0$~and
  its differential at the point~$(0,0)$~is
  \[
    (x_1,\cdots,x_k; z) \mapsto x_1 \X_1(0) +
    \cdots+ x_k \X_k(0) + \sum_j z_j \frac{\partial}{\partial y_j} (0).
  \]
  Therefore it is invertible. By Proposition
  \ref{ExistenceInverse}~$g$~admits a formal inverse~$h \in \Q_p \left[ \left[ t_1, \cdots, t_k, \y \right] \right]^d$
  with a radius of convergence~$r_h >0$. Denote by~$\z$~the set of coordinates~$(t_1, \cdots, t_k, y_1, \cdots,
  y_{d-k})$. Pick integers~$K, L$~such that~$\abs p ^K < r_g$~and~$\abs p^L < r_h$~such that~$g(B(0, \abs p^K)) \subset
  B(0, \abs p^L)$. Let $\mathcal V$ denote $g(B(0, \abs p^K))$; it is a clopen subset of~$\Z_p^d$~because~$B(0, \abs
  p^K)$~is clopen. Set~$\varphi := \frac{1}{p^L} g(p^K \z)$~and~$\psi := \frac{1}{p^K} h(p^L \z)$, they both belong to
 ~$\Q_p \langle \z \rangle^d$~and are inverse of each other and we have~$\varphi^* \X_i = \partial_i$. Finally, since
 ~$\varphi \in \Q_p \langle \z \rangle^d$, the map~$\varphi^*$~preserves Tate-analytic vector fields.
\end{proof}

\begin{thm}[$p$-adic version of \cite{epstein1979transformation} Theorem 1.1]\label{theoremPAdicEpsteinThurston}
    Let~$\h$~be a nilpotent Lie algebra of Tate-analytic vector fields of~$\Z_p^d$, then~$d \geq \dl (\h)$.
  \end{thm}

  \begin{proof}
    We follow the proof of \cite{cantat2014mapping} Proposition 3.10 and proceed by induction on the dimension~$d$. If
    ~$d=0$, there is nothing to prove. Suppose~$d \geq 1$~and that the result is true in dimension~$d-1$. Since~$\h$~is
    nilpotent, its center is not trivial.
    Let~$\X$~be a nonzero central element of~$\h$. Let~$m$~be a point where~$\X(m) \neq 0$, then by Theorem
    \ref{pAdicFrob}, there exists a small clopen subset~$\mathcal V \subset \Z_p^d$~and an analytic diffeomorphism~$\varphi:
    \mathcal V \rightarrow \Z_p^d$~that yields coordinates~$x_1, \cdots, x_d$~over~$\mathcal V$~such that~$\varphi_* \X =
    \partial_d$~and such that~$\varphi_*$~maps Tate-analytic vector fields to Tate-analytic vector fields.
    By Proposition \ref{AnalyticContinuation} the morphism of restriction~$\h \rightarrow \h_{|\mathcal V}$~is an
    isomorphism of Lie algebras. We replace~$\h$~by~$\h_{|\mathcal V}$~and work with the coordinates~$x_1, \cdots, x_d$~over
   ~$\mathcal V$. Every vector field~$\mathbf Y$~of~$\h$~must commute
   with~$\X = \partial_d$~so it is of the form \[ \mathbf Y = \sum_{i=1}^d u_i(x_1,\cdots, x_{d-1}) \partial_i. \] Let~$\pi:
    \mathcal V \simeq \Z_p^d \rightarrow \Z_p^{d-1}$~be the projection over the first~$d-1$~coordinates. This yields a Lie algebra
    homomorphism~$\pi_* : \h \rightarrow \Theta(\Z_p^{d-1})$. Denote by~$\h_1$~the image of~$\h$~under~$\pi_*$~and~$\h_0$
    its kernel. We have the exact sequence \[ 0 \rightarrow \h_0 \rightarrow \h \rightarrow \h_1 \rightarrow 0. \]
    Now,~$\h_0$~consists of Tate-analytic vector fields of~$\h$~of the form~$u(x_1, \ldots, x_{d-1}) \partial_d$~so it
    is abelian and~$\h_1$~is nilpotent because~$\h$~is. So we get~$\dl(\h) \leq \dl(\h_1) + 1$~by the exact sequence
    and~$\dl(\h_1) \leq d-1$~by induction.
  \end{proof}

  We discuss the optimality of Theorem \ref{theoremPAdicEpsteinThurston} in Section \ref{SubSecOptimality}.

\paragraph{The theorem of Bell and Poonen.--}

The following theorem first proven by Bell in \cite{Bell05} then by Poonen in \cite{poonen2014p} gives us
an easy way to construct flows from analytic transformations. This is a very strong theorem as it shows
that, contrary to~$\R$, over~$\Q_p$~a lot of analytic diffeomorphisms are in a flow. See
\cite{Cantat_smf_18} for a more precise discussion on Bell-Poonen theorem.

\begin{thm}[Bell-Poonen]\label{theoremBellPoonen}
  Let~$d \geq 1$~be an integer, and~$f \in \Z_p \langle \x \rangle^d$. Take~$c > \frac{1}{p-1}$
  and suppose that~$f \equiv \id \mod p^c$, then
  \begin{enumerate}
    \item~$f$~is a Tate-analytic diffeomorphism.
    \item There exists a unique Tate-analytic flow~$\Phi \in \Z_p \langle \x,t \rangle^d$~such that
      \[ \forall n \in \Z, \quad \Phi(\x,n) = f^n (\x). \]
      In particular,~$\Phi_1 = f$.
  \end{enumerate}
\end{thm}

In fact, Poonen showed this theorem for the valuation ring of any ultrametric field~$\K$. So,
Bell-Poonen Theorem also holds over~$\D_p$~or over any finite extension of~$\Q_p$~for example.

\begin{cor}\label{NoTorsion}
  Let~$H$~be a subgroup of~$\Diff^{an}_1 (\Z_p^d)$~with~$p \geq 3$, then~$H$~is torsion-free.
\end{cor}

\begin{proof}
  Let~$h \in H$, suppose that~$h$~has order~$N < \infty$. By Theorem \ref{theoremBellPoonen}, there exists an
  Tate-analytic flow~$\Phi$~such that~$\Phi_1 = h$. Then for all
 ~$\x \in \Z_p^d$~the function~$t \in \Z_p \mapsto \Phi_t (\x) - \x \in \Z_p^d$~is analytic and has an infinite number of
  zeros, so it is zero everywhere by Proposition \ref{AnalyticContinuation}. Therefore~$\Phi_1 (\x) = h(\x) = \x$~and~$h =
  \id$.
\end{proof}

The next proposition won't be used in the proof of Theorem \ref{BoundNilpotentGroups} but it gives useful
information on the dynamics of Tate-analytic flows.

\begin{prop}\label{StableAIndiceFiniPres}
  Let~$\Phi \in \Z_p \langle \x,t \rangle$~be a Tate-analytic flow over~$\Z_p^d$.
  If~$\U \subset \Z_p^d$~is a clopen set, then there exists an~$\epsilon >0$~such that \[ \forall t \in \Z_p,
  \quad \abs t \leq \epsilon \Rightarrow \Phi_t (\U) = \U. \]
\end{prop}

\begin{proof}
  Fix~$x \in \Z_p^d$~and~$0 < r \leq 1$. Since~$\Phi_t \rightarrow
  \id$~as~$t \rightarrow 0$~in~$\Diff^{an} (\Z_p)$, there exists~$\epsilon >0$~such that for all~$t \in \Z_p$,~$\abs t
  \leq \epsilon \Rightarrow \norm{\Phi_t - \id} \leq r$. Now for all~$z \in \Z_p^d, \norm{ \Phi_t (z) - z}
  \leq \norm{\Phi_t - \id} \leq r$. Then, for all~$y$~such that~$\norm{y -x} \leq r$,

  \begin{align*}
\norm{\Phi_t (y) - x}
&= \norm{\Phi_t (y) - y + y -x } \\ &\leq \max( \norm{\Phi_t(y) - y}, \norm{y-x} ) \leq r.
  \end{align*}

  So if~$\abs t \leq \epsilon$, we have~$\Phi_t (B(x,r)) \subset B(x,r)$~and~$\Phi_{-t}(B(x,r)) \subset B(x,r)$, so
  we get the equality.

  Since~$\U$~is clopen, by compactness,~$\U = \bigcup_{i=1}^T B(x_i, r_i)$~for some finite set~$\left\{
  x_1, \cdots, x_T \right\} \subset \U$~and radii~$r_i \in (0, 1]$. Thus, the results follows from the case of one
ball.  \end{proof}

\subsection{Infinite-dimensional analytic manifold over~$\Q_p$}\label{SecInfiniteDimensionalAnalyticmanifolds}
The main goal of the next two sections is to show that the topological group~$\Diff^{an}(\Z_p^d)$~is in fact an infinite
dimensional Lie group over~$\Q_p$.

We refer to \cite{bourbaki2007varietes} for reference on analytic functions and analytic manifolds over a Banach
space. In this section~$\k$~is an ultrametric complete field and~$E, F$~are Banach spaces over~$\k$~(potentially of infinite
dimension). As we shall see,
taking~$\k  = \Q_p$~and~$E, F = \Q_p^d$~allows one to recover the definition of converging power series and analytic
functions over~$\Q_p^d$.

Basically, if~$A$~is a Banach algebra over~$\Q_p$, then any map of the form~$f: A^d \rightarrow A$~such that locally at
any point~$x \in A^d$, there is a expression of~$f$~as a converging power series
\[ f(x + h) = \sum_{I \subset \Z_+^d} a_I h^I \]
with~$a_I \in A, a_I \rightarrow 0$~is an analytic map from~$A^d$~to~$A$. The problem is that if~$A$~is not finite dimensional,
this definition is not enough, as for example a continuous linear map is not necessarily described by an expression of
this form but still should be analytic.

\paragraph{Multi-indices, multi-linear maps.--}
If~$\alpha = (\alpha_1, \cdots, \alpha_d) \in \Z_+^d$~is a multi-index, then~$\abs \alpha := \sum_i \alpha_i$. For~$1 \leq
j \leq \abs \alpha$, we define
\[ \alpha(j) = \max \left\{ k + 1 \in \Z_+ : \alpha_1 + \cdots + \alpha_{k} < j \right\}. \]
The sequence~$(\alpha(j))_{1 \leq j \leq \abs \alpha}$~is the increasing sequence consisting of~$\alpha_1$~times the
number 1,~$\alpha_2$~times the number 2, \ldots,~$\alpha_d$~times the number~$d$. For example, if~$\alpha = (1, 5 ,7)$,
then~$d=3, \abs \alpha = 13$~and
\[ (\alpha(j))_{1 \leq j \leq 13} = (1, 2, 2, 2, 2, 2, 3, 3, 3, 3, 3, 3, 3).
\]
For~$1 \leq i \leq d$, we denote by
$p_i: E^d \rightarrow E$~the projection to the~$i$-th coordinate. For a multi-index~$\alpha \in \Z_+^d$, we define
\[
p_\alpha := (p_{\alpha(j)})_{1 \leq j \leq \abs \alpha}: E^d \rightarrow E^{\abs \alpha}.
\]

If~$\beta \in \Z_+^d$~is another multi-index, then we write~$\alpha + \beta$~for the multi-index~$(\alpha_i +
\beta_i)_{1 \leq i \leq d}$.
We write~$\alpha \geq \beta$~if~$\alpha_i \geq \beta_i$~for all~$1 \leq i \leq d$; in that case there is a unique
multi-index~$\gamma$~such that~$\alpha = \beta + \gamma$, and we set~$\alpha - \beta := \gamma$. We also define the
binomial coefficient~$\binom{\alpha}{\beta} := \binom{\alpha_1}{\beta_1} \cdots \binom{\alpha_d}{\beta_d}$. Finally, if
$\x = (x_1, \cdots, x_d)$, then~$\x^\alpha := x_1^{\alpha_1} \cdots x_d^{\alpha_d}$~ and if~$\mathbf y = (y_1, \cdots,
y_d)$, one has the identity
\begin{align*}
  (\x + \mathbf y)^\alpha &= (x_1 + y_1)^{\alpha_1} \cdots (x_d + y_d)^{\alpha_d}  \\
  &= \left(\sum_{\beta_1 = 0}^{\alpha_1} \binom{\alpha_1}{\beta_1}x_1^{\beta_1} y_1^{\alpha_1 - \beta_1}   \right) \cdots
  \left( \sum_{\beta_d = 0}^{\alpha_d} \binom{\alpha_d}{\beta_d}x_d^{\beta_d} y_d^{\alpha_d - \beta_d}   \right) \\
  &= \sum_{0 \leq \beta_1 \leq \alpha_1} \cdots \sum_{0 \leq \beta_d \leq \alpha_d} \binom{\alpha_1}{\beta_1} \cdots
  \binom{\alpha_d}{\beta_d} x_1^{\beta_1} \cdots x_d^{\beta_d} y_1^{\alpha_1 - \beta_1} \cdots y_d^{\alpha_d - \beta_d} \\
  &= \sum_{\beta \leq \alpha }\binom{\alpha}{\beta} \x^{\beta} \mathbf y^{\alpha - \beta}.
\end{align*}
For an integer~$k$, let~$\mathcal L_k (E, F)$~be the set of continuous multilinear maps from~$E^k$~to~$F$~equipped with
the topology of uniform convergence over bounded subsets. The norm of an element~$\phi \in \mathcal L_k(E,F)$~is defined by
\[ \norm \phi = \inf \left\{ a > 0 : \forall x_1, \cdots, x_k \in E^k, \norm{\phi(x_1, \cdots, x_k)}_F \leq a \norm{x_1}_E \cdots \norm{x_k}_E \right\}. \]

\paragraph*{Continuous polynomial maps and power series.--} (\cite{bourbaki2007varietes} Appendix of \S 1-7)  A
\emph{continuous homogeneous polynomial map of multi degree}~$\alpha$, is a map~$f: E^d \rightarrow F$
such that there exists~$u \in \mathcal L_{\abs{\alpha}}(E,F)$~for which~$f = u \circ p_\alpha$.
We denote by~$P_\alpha(E,F)$~the vector space of continuous homogeneous polynomial maps of multi-degree~$\alpha$~equipped
with the quotient topology from~$\mathcal L_{\abs {\alpha}}(E,F)$. The norm of a continuous homogeneous polynomial map~$P \in P_\alpha
(E,F)$~is defined by
\[ \norm P := \inf_{u \in \mathcal L_{\abs{\alpha}} (E,F), P = u \circ p_\alpha } \norm u_{\mathcal L_{\abs{\alpha}}(E,F)}. \]

\begin{ex}
  Set~$E, F = \Q_p \langle \x \rangle$. Let~$P$~be the monomial~$\x^\alpha$, then the map~$P: g \in \Q_p
  \langle \x \rangle^d \mapsto P(g) \in \Q_p \langle \x \rangle$~is a
  continuous homogeneous polynomial map of multi-degree~$\alpha$. Indeed, let~$k = \abs \alpha$~and
  consider the multilinear map
  \[
  \begin{array}{lccc}
  {T_{k}}:& {E^{k}}& \longrightarrow & {F}\\
  &{(f_1, \cdots, f_{k})} &\longmapsto &{f_1 \cdots f_{k}};
  \end{array}
\]
  it is continuous as~$\norm{ T_{k}(f_1, \cdots, f_{k})} \leq \norm{f_1} \cdots \norm{f_{k}}$~and~$~P = T_{k} \circ
  p_\alpha$.

  Furthermore, for a multi-index~$\beta$, define~$\phi_\beta: \Q_p \langle \x \rangle \rightarrow \Q_p \langle \x
  \rangle$~such that~$\phi_\beta (g)$~is the homogeneous part of multi-degree~$\beta$~of~$g$. Then,~$\phi_\beta$~is
  linear and continuous, therefore if~$P(\x) = \x^\alpha$, the map~$g \in \Q_p \langle \x \rangle^d \mapsto
  P(\phi_{\beta_1}(g_1), \cdots, \phi_{\beta_d}(g_d))$~is a continuous homogeneous polynomial map of
  multi-degree~$\alpha$~for any multi-index $(\beta_i)_{1 \leq i \leq d}$.
\end{ex}

For an integer~$k$,~$P_k (E^d,F)$~is the direct sum of the~$P_\alpha(E,F)$~for~$\alpha$~such that~$\abs \alpha = k$, the
elements of~$P_k (E^d,F)$~are the \emph{continuous homogeneous polynomial maps of total degree~$k$.}

\begin{ex}\label{ExampleOfhomogeneousContinuousPolynomialMap}
  If~$P \in \Q_p [\x]$~is a homogeneous polynomial of degree~$k$~in~$d$~variables, then the map~$P: g \in \Q_p \langle \x \rangle^d
  \mapsto P(g)$~is a continuous homogeneous polynomial map of total degree~$k$~and for any sequence of multi-index
  ~$(\beta_i)_{1 \leq i \leq d}$, the map~$g \in \Q_p \langle \x \rangle^d \mapsto P(\phi_{\beta_1}(g), \cdots,
  \phi_{\beta_d}(g))$~also is.
\end{ex}

We denote by~$P(E^d, F)$~the direct sum of the spaces~$P(E^d, F)$, its elements are the \emph{continuous polynomial maps in
$d$~variables}.

\begin{prop}
  Set~$E, F = \Q_p \langle \x \rangle$. Take a polynomial~$P \in \Q_p [\x ]$. Then,~$P$~induces a
  continuous polynomial map~$E^d \rightarrow F$~and the linear embedding~$\Q_p[\x] \hookrightarrow P(E^d, F)$~is an isometry.
\end{prop}

Finally, the set~$\hat P (E^d, F)$~of \emph{power series} in~$d$~variables over~$E$~is the (infinite) product of the~$P_\alpha(E,
F)$~(or of the~$P_k(E^d, F)$) for~$\alpha \in \Z_+^d$~(for~$k \in \Z_+$) equipped with the product topology of the discrete topology
over each factor; equivalently if~$f = \sum_\alpha f_\alpha \in \hat P(E^d, F)$, then the order of vanishing at~$0$~of
$f$~is~$\ord (f) = \min \left\{ \abs \alpha : f_\alpha \neq 0 \right\}$~and this is the topology induced by the norm~$\norm f :=
2^{- \ord(f)}$. The
space~$\hat P(E^d, F)$~is complete Hausdorff for this topology. A \emph{converging power
series} is an element~$f = \sum_\alpha f_\alpha$~of~$\hat P(E^d, F)$~such that there exists~$R \in (\R_{>0})^d$
satisfying
$\sup_\alpha R^{\alpha} \norm{f_\alpha}_{P_\alpha(E,F)} < + \infty$. If~$f = \sum_\alpha f_\alpha$, then the \emph{polyradius of
convergence of~$f$} is
\[
  r(f) := \sup \left\{ R \in (\R_{>0})^d : R^{\alpha} \norm{f_\alpha} \rightarrow 0 \text{ when } \abs
  \alpha \rightarrow \infty \right\}.
\]

\begin{dfn}
  Let~$\mathcal U$~be an open subset of~$E^d$, a map~$f: \mathcal U \rightarrow F$~is \emph{analytic} at a point~$a \in
  \mathcal U$~if there exists a converging power series~$f_a$~such that for all~$x$~in a small neighbourhood of~$a$~in
  ~$\mathcal U, f(a +x) = f_a (x)$. The function~$f$~is analytic if it is analytic at every point of~$\mathcal U$.

  For any integer~$m \geq 1$, a map~$f: \mathcal U \rightarrow F^m$~is analytic if each of its coordinates is analytic.
\end{dfn}

\begin{ex}
  Every continuous linear map~$\Q_p \langle \x \rangle^d \rightarrow \Q_p \langle \x \rangle^d$~is analytic.
\end{ex}

\begin{prop}\label{PropCompositionIsAnalytic}
  The map~$~\text{\emph{Comp}}: (h, f) \in \Z_p \langle \x \rangle^d \times \Z_p \langle \x \rangle^d \mapsto h \circ f \in
  \Z_p \langle \x \rangle^d$~is analytic. In particular, it is linear in~$h$.
\end{prop}

\begin{proof}
  It is enough to show that the map~$\Phi: (h,f) \in \Z_p \langle \x \rangle \times \Z_p \langle \x \rangle^d \mapsto h
  \circ f \in \Z_p \langle \x \rangle$~is analytic.
  Let~$(h, f) \in \Z_p \langle \x \rangle \times \Z_p \langle \x \rangle^d$, we show that~$\Phi$~is analytic at~$(h,
  f)$. Let~$g \in \Z_p \langle \x \rangle^d$~and write~$h(\x) = \sum_\alpha a_\alpha \x^\alpha$, then
  \begin{align*}
    h \circ (f + g(\x)) &= \sum_\alpha a_\alpha (f(\x) + g(\x))^\alpha \\
    &= \sum_\alpha \sum_{\gamma \leq \alpha } a_\alpha \binom{\alpha}{\gamma} f(\x)^{\alpha - \gamma} g(\x)^{\gamma} \\
    &= \sum_{\beta} \left( \sum_{\alpha \geq \beta} a_\alpha \binom{\alpha}{\beta}f(\x)^{\alpha - \beta} \right) g(\x)^\beta \\
    &= \sum_\beta Q_{\beta,f}(h)(\x) \cdot g(\x)^\beta
  \end{align*}
  where~$Q_{\beta,f}: \Q_p \langle \x \rangle \rightarrow \Q_p \langle \x \rangle$~is a continuous linear
  map and $\norm {Q_\beta} \rightarrow 0$ when $\beta \rightarrow \infty$, this is a
  converging power series in the variables~$(h, g)$~of polyradius of convergence~$(+ \infty, 1)$. Therefore~$\Phi$~is
  analytic at any point~$(0, f)$~and by linearity in~$h, \Phi$~is analytic at any point~$(h,f)$.
\end{proof}

\paragraph{Analytic manifolds.--}
Let~$\K$~be an ultrametric field and let~$X$~be a topological space. A~$\K$-\emph{chart} of~$X$~is a
homeomorphism~$\phi: U \rightarrow \phi(U) \subset E$~where~$U$~in an open subset of~$X$~and~$E$~a Banach space over~$\K$. We say that two
$\K$-charts~$\phi: U \rightarrow E, \psi: V \rightarrow F$~are \emph{compatible} if
\begin{enumerate}
  \item~$\phi (U \cap V)$~is open in~$E$~and~$\psi (U \cap V)$~is open in~$F$.
  \item~$\psi \circ \inv \phi : \phi (U \cap V) \rightarrow F$~is analytic.
  \item~$\phi \circ \inv \psi: \psi (U \cap V) \rightarrow E$~is analytic.
\end{enumerate}

An analytic  manifold~$X$~over~$\K$~is defined classically as a topological space equipped with
an atlas of compatible~$\K$-charts. For a
point~$x \in X$, the tangent space at~$x$~is denoted by~$T_x X$. A
function~$f: X \rightarrow Y$~between two analytic manifolds is analytic if for every chart~$\phi: U \subset X
\rightarrow E, \psi: V \subset Y \rightarrow F$, the map~$\psi \circ f \circ \inv \phi: \inv \phi(U) \rightarrow F$~is
analytic. The differential of~$f$~at a point~$x$~will be denoted~$D_x f$.

\begin{prop}\label{PropDiffAnIsAnAnalyticVariety}
  The topological space~$\Diff^{an}(\Z_p^d)$~is an analytic manifold over~$\Q_p$, it is in fact an open subset of the
  Banach space~$\Q_p \langle \x \rangle^d$. The subgroups~$\Diff^{an}_c(\Z_p^d)$
  for~$c > \frac{1}{p-1}$~are diffeomorphic to~$\Z_p \langle \x \rangle^d$~and they form a basis of neighbourhood of
 ~$\id$~in~$\Diff^{an}(\Z_p^d)$. \end{prop}

\begin{proof}
  Theorem \ref{theoremBellPoonen} shows that~$\Diff^{an}_c (\Z_p^d)$~is the ball of center~$\id$~and radius~$\abs p^c$~in~$\Z_p
  \langle \x \rangle^d$, using Proposition \ref{truc1} we see that for every~$f \in \Diff^{an}(\Z_p^d)$, the ball of
  center~$f$~and radius~$\abs p^c$~is included in~$f \circ \Diff^{an}_c(\Z_p^d)$~therefore it is an open set of~$\Q_p \langle \x
  \rangle^d$, so~$\Diff^{an}(\Z_p^d)$~is an infinite dimensional analytic manifold over~$\Q_p$.
\end{proof}

\paragraph{The implicit function theorem.--}
Let~$X, Y, Z$~be manifolds over~$\K$~and let~$f: X \times Y \rightarrow Z$~be an analytic map. Let~$(a,b) \in X \times Y$, we write~$D_{(a,b)} f$~the differential map of~$f$~at~$(a,b)$~and
let~$D_{(a,b)}^{(1)} f$~be the differential of the partial map~$x \in X \mapsto f(x, b)$~at~$a$~and~$D_{(a,b)}^{(2)} f$~the
differential of the partial map~$y \in Y \mapsto f(a, y)$~at~$b$. Then, one has~$T_{(a,b)} X \times Y = T_a X \times T_b Y$
and~$D_{(a,b)}f (u,v) = D_{(a,b)}^{(1)}f \cdot u + D_{(a,b)}^{(2)} f \cdot v$.

\begin{thm}[Implicit function theorem, 5.6.1 of \cite{bourbaki2007varietes}]
  Suppose that~$D_{(a,b)}^{(2)}f$~is bijective, then there exists an open neighbourhood~$U$~of~$a$~in~$X$~and an open
  neighbourhood~$V$~of~$b$~in~$Y$~and a unique analytic map~$g: U \rightarrow V$~such that
  \[ \forall x \in U, \quad f(x, g(x)) = f(a,b) \]
  and the differential of~$g$~at any~$x \in \mathcal U$~is given by
  \[ D_x g = - \inv{\left( D_{(x,g(x))}^{(2)} f \right)} \circ D_{(x,g(x))}^{(1)} f \]
\end{thm}

\begin{prop}\label{PropInvIsAnalytic}
  The inversion map~$\Inv : f \in \Diff^{an}(\Z_p^d) \mapsto \inv f$~is analytic.
\end{prop}

\begin{proof}
  We write~$\mathcal U = \Diff^{an}(\Z_p^d)$, we know that~$\U$~is an analytic manifold over~$\Q_p$~by
  Proposition \ref{PropDiffAnIsAnAnalyticVariety}.
   By Proposition \ref{PropCompositionIsAnalytic}, the
  composition operation is analytic over~$\Z_p \langle \x \rangle^d \times \Z_p \langle \x \rangle^d$, therefore it is
  over~$\mathcal U \times \mathcal U$.

  To show that~$\Inv$~is analytic we only need to show that it is
  analytic at~$\id$. Indeed, take~$f \in \mathcal U$, then~$\Inv = L_{\inv f} \circ \Inv
  \circ R_{\inv  f}$~where~$R_{\inv f}$~is composition on the right by~${\inv f}$~and~$L_{\inv f}$~composition on
  the left. Since~$L_{\inv f}$ and $R_{\inv f}$ are analytic,~$\Inv$~is analytic at~$f$~if and only if it is analytic at~$\id$.
  To show that~$\Inv$~is analytic at~$\id$, we use the implicit function theorem, since the map~$M: (f,g)
  \in \U \times \U \rightarrow f \circ g \in \U$
  is analytic and the partial differential~$D_{\id, \id}^{(2)} M = \id$, one has the existence of a unique function~$G:
  \mathcal V \rightarrow \mathcal \U$~with~$\mathcal V$~an open neighbourhood of~$\id$~such that~$G$~is analytic at
  ~$\id$~and~$M(f, G(f)) = \id$~for all~$f \in \mathcal V$. Therefore~$\Inv_{|\mathcal V} = G$~and inversion is analytic
  at~$\id$.
\end{proof}

\subsection{$p$-adic Lie groups}
We refer to \cite{Bourbaki06} for more details on the results provided in this section.

A~$p$-adic Lie group~$G$~is a topological group with a structure of a~$p$-adic analytic manifold such that the
multiplication map and the inverse map are analytic. The dimension of~$G$~is its dimension as an analytic manifold. It
can  be infinite. Its \emph{Lie algebra}~$\g$~is the tangent space of~$G$~at the neutral element, it is equipped with a
Lie bracket~$[\cdot , \cdot]$~defined as follows. Let~$g \in G$~and~$\iota_g: h \in G \mapsto g h \inv g$, then~$\Ad (g)
:= D_e \iota_g \in \GL(\g)$~is the adjoint representation of~$G$. Define~$\ad := D_e \Ad$, then
\[
  \forall \X, \Y \in \g, [\X, \Y] := \ad(\X)(\Y).
\]

\begin{thm}\label{theoremDiffAnIsALieGroup}
  The topological group~$\Diff^{an} (\Z_p^d)$~is an infinite-dimensional Lie group over~$\Q_p$. Its Lie
  Algebra is~$\Theta(\Z_p^d)$.

  Moreover, the subgroups~$\Diff^{an}_c (\Z_p^d)$~are also Lie groups for~$c > \frac{1}{p-1}$~and they form a basis of
  neighbourhood of~$\id$~in~$\Diff^{an}_c(\Z_p^d)$.
\end{thm}

\begin{proof}
  The fact that~$\Diff^{an} (\Z_p^d)$~is a Lie group over~$\Q_p$~follows from Propositions
  \ref{PropCompositionIsAnalytic}, \ref{PropDiffAnIsAnAnalyticVariety} and \ref{PropInvIsAnalytic} where
  it was shown that it was an analytic manifold and that composition and inversion are analytic maps. The statement for
 ~$\Diff^{an}_c(\Z_p^d)$~follows from the same propositions.

  The tangent space at~$\id$~is~$\Q_p \langle \x \rangle^d~$~that we identify with~$\Theta(\Z_p^d)$~and under this
  identification the Lie bracket between two Tate-analytic vector fields corresponds to the Lie bracket of the Lie
  algebra of the Lie group~$\Diff^{an}(\Z_p^d)$~because if~$\X, \Y$~are of norm~$\leq \abs p^c$ with $c >
  \frac{1}{p-1}$, then they admit global
  Tate-analytic flows~$\Phi^\X$~and~$\Phi^\Y$~by Proposition \ref{PropExistenceGlobalTateAnalyticFlow} and

  \begin{eqnarray*}
    [\X, \Y] &= \frac{\partial }{\partial_s}_{|s=0} \frac{\partial }{\partial t}_{|t=0} \Phi^\X_{-s} \circ \Phi^\Y_t
    \circ \Phi^\X_s \\
    &= \frac{\partial }{\partial_s}_{|s=0} \frac{\partial }{\partial t}_{|t=0} \iota_{\Phi^\X_s}
    (\Phi^\Y_t) \\
    &=  D_{\id} \Ad(\X) (\Y) = \ad(\X) (\Y).
  \end{eqnarray*}

  On the other hand, if~$f,g \in \Diff^{an}_c(\Z_p^d)$~with~$c > \frac{1}{p-1}$
  , then~$\frac{\partial}{\partial s}_{|s=0} \frac{\partial}{\partial t}_{|t=0}
  \Phi^f_{-s} \circ \Phi^g_t \circ \Phi^f_s = [\X_f, \X_g] = \ad\X_f (\X_g)$.
\end{proof}

\begin{rmq}
  Since Bell-Poonen theorem holds for any ultrametric field, the same proof shows that $\Diff^{an}(\D_p^d)$ is a Lie
  group over~$\C_p$. In fact, for any complete extension~$\K$~of~$\Q_p$~with unit ball~$\mathbf A$, the
  group~$\Diff^{an}(\mathbf A^d)$~is a Lie group over~$\K$.
\end{rmq}

\begin{thm}[\cite{Bourbaki06}, \S 8, Theorem 1]\label{PropContinousMorphismIsAnalytic}
  Let~$G, H$~be Lie groups over~$\Q_p$~and~$\phi: G \rightarrow H$~be a continuous homomorphism of topological groups.
  Then,~$\phi$~is analytic and therefore a homomorphism of Lie groups.
\end{thm}

\begin{rmq}
  The proof relies heavily on~$\Q$~being dense in~$\Q_p$~and the theorem is false if we replace
 ~$\Q_p$~by any finite extension of~$\Q_p$. Indeed, suppose for example that~$K = \Q_p (\sqrt \alpha)$~is
  a quadratic extension. Any element~$z$~of~$\K$~is of the form~$z = x + \sqrt \alpha y$. Then, the function
  \[ f: z = x + \sqrt \alpha y \mapsto x - \sqrt \alpha y \]
  is a continuous group homomorphism, it is~$\Q_p$-analytic but not~$\K$-analytic as~$f_{|1 \cdot
  \Q_p} = \id$~and~$f_{|\sqrt \alpha \cdot \Q_p} = - \id$.
\end{rmq}

Let~$\Gamma$~be a finitely generated group, the pro-$p$ completion~$\Gamma_p$~of~$\Gamma$~is the projective limit of the
quotient of~$\Gamma$~that are finite~$p$-groups, it is a topological group with respect to the profinite topology. In
particular, for any~$\gamma \in \Gamma$, the group homomorphism~$n
\in \Z \mapsto \gamma^n \in \Gamma$~extends uniquely to a continuous group homomorphism~$t \in \Z_p \mapsto \gamma^t \in
\Gamma_p$. In the context of Tate-analytic diffeomorphisms, if~$p \geq 3$~and~$f \equiv \id \mod p$, then the extension
$n \in \Z \mapsto f^n \in \Diff^{an}_1 (\Z_p^d)$~is the Tate-analytic flow~$t \in \Z_p \mapsto \Phi_t^f \in
\Diff^{an}(\Z_p^d)$~associated to~$f$~given by Bell-Poonen theorem.

\begin{prop}\label{PropEmbeddingOfProPcompletion}
  Let~$p$~be a prime, let $c>0$ be such that $c > \frac{1}{p-1}$ and let~$G$~be a compact Lie group over~$\Q_p$. Let
  $\Gamma$~be a finitely generated subgroup of~$G$~such that~$G$~is the pro-$p$-completion of~$\Gamma$ and let~$\iota:
  \Gamma \rightarrow \Diff^{an}_c (\Z_p^d)$~be a group homomorphism, then~$\iota$~extends
  uniquely to a Lie group homomorphism~$\iota: G \rightarrow \Diff^{an}_c (\Z_p^d)$~such that for all~$t \in \Z_p$,
  all~$g \in \Gamma$,~$\iota(g^t) = \iota(g)^t$~and the map~$(t, \x) \in \Z_p \times \Z_p^d \mapsto \iota(g)^t (\x)$~is
  analytic.
\end{prop}

\begin{proof}
  Theorem 2.11 of \cite{cantat2014algebraic} shows that~$\iota$~extends uniquely to a continuous map. In
  \cite{cantat2014algebraic} this is only shown when $p \geq 3$ and $c = 1$ but the proof is identical
  with $p \geq 2$ and $c > \frac{1}{p-1}$ at it is only required that the image of the elements of $\Gamma$ admits a
  Tate-analytic flow.
  Since~$G$~and~$\Diff^{an}_c (\Z_p^d)$~are both Lie groups over~$\Q_p$,~$\iota$~is automatically a Lie
  group homomorphism by Theorem \ref{PropContinousMorphismIsAnalytic}.
\end{proof}

\begin{thm}[\cite{Bourbaki06}, \S 8, Theorem 2]\label{theoremClosedSubgroupAreLieGroups}
  Let~$G$~be a finite-dimensional Lie group over~$\Q_p$, then every closed subgroup of~$G$~is a Lie subgroup of~$G$.
\end{thm}

\begin{prop}[\cite{Bourbaki06}, \S 9, Corollary of Proposition 6]
  \label{PropOpenSubgroupDerivedSeries}
  Let~$G$~be a finite-dimensional Lie group over~$\Q_p$~and $\mathfrak g$ its Lie algebra, there exists an open subgroup
  ~$G_0$~of~$G$~such that for all~$i \geq 0$, the subgroups~$D^i(G_0)$~and~$D_i(G_0)$ are Lie subgroups with Lie algebra
  ~$\mathcal D^i (\h)$~and~$\mathcal D_i (\h)$~respectively.
\end{prop}

\subsection{Nilpotent groups and embedding into~$p$-adic Lie groups.}

\subsubsection{Nilpotent groups}\label{SubSubSecNilpotentGroups}
The main goal of this section is to show that if~$H$~is a finitely generated nilpotent group with generators~$h_1,
\ldots, h_s$, then for any~$m \geq 1$~the subgroup~$H_m$~of~$H$~generated by~$h_1^m, \ldots, h_s^m$~is a finite index subgroup
of~$H$. This will be useful in the proof of Theorem \ref{BoundNilpotentGroups} because if~$H \subset
\Diff^{an}_1 (\Z_p^d)$~we will need to consider such a subgroup~$H_m$~to get the desired result.

Recall the notation introduced in \S~\ref{par:nilpotent_and_solvable} for nilpotent and solvable groups
and Lie algebras. We shall say that an expression that involves~$k$~commutator brackets is a commutator of
length~$k$; for instance~$[ [a, [b,c]], d]$~is a commutator of length 3 and a single element can be viewed
as a commutator of length 0. For~$k \geq 1$, we denote by~$[a_1; \cdots; a_k]$~the commutator~$[a_1, [a_2,
  \cdots, [a_{k-1}, a_k] \cdots]$; its length is~$k$.

Let~$G, G', G''$~be groups, a map~$\phi: G \times G' \rightarrow G''$~is \emph{bilinear} if for every~$g \in G, g' \in
G'$, the maps~$\phi(g, \cdot)$~and $\phi( \cdot, g')$ are group homomorphisms. More generally, a map~$G_1 \times \cdots
\times G_m \rightarrow G$~is~$m$-linear if fixing~$m-1$~coordinates yields a group homomorphism. For any
triple of elements~$x,y,z$~in~$G$, we have

  \begin{itemize}
    \item~$\inv{[x,y]} = [y,x]$.
    \item~$[x,yz] = [x,y] [ y, [x,z]] [x,z]$.
    \item~$[xy,z] = [x,[y,z]] [y,z] [x,z]$.
  \end{itemize}

  The image of the map~$(a,b) \mapsto [a,b]$~from~$G \times D^{k-1}(G)$~to~$D^k(G)$~generates $D^k (G)$. It follows from the
  last three formulas that, for every~$k \geq 1$, this map induces a bilinear map
  \[ \mathrm{co}_k : G \times D^{k-1}(G) \mapsto D^k(G) / D^{k+1} (G) \]

  and the image $\im \mathrm{co}_k$ generates $D^k(G) / D^{k+1}(G)$.

  \begin{prop}\label{PropGeneratorsOfLowerCentralSeries}
  Let~$G$~be a group and~$S$~a set of generators of~$G$.
  \begin{enumerate}
    \item for every integer~$k \geq 0$, the subgroup~$D^{k}(G) / D^{k+1}(G)$~is generated by the
      commutators of length~$k$~consisting of elements of~$S$.
    \item if~$G$~is finitely generated, then~$D^{k}(G) / D^{k+1}(G)$~is finitely generated for every~$k
      \geq 0$.
    \item If~$G$~is nilpotent, then~$D^{\nilp(G)-1} (G)$~is generated by the commutators of length~$\nilp(G) -1$~in
      elements of~$S$.

  \end{enumerate}
   \end{prop}

\begin{proof}
  Let us prove the first assertion by induction on~$k$. Let~$X_k$~be the set of commutators of length~$k$
  in elements of~$S$. The initialization~$k=0$~follows from~$X_0 = S$~and the fact that~$S$~generates
 ~$G$. Now, suppose~$k \geq 1$~and that~$X_{k-1}$~generates~$D^{k-1} (G) /
 D^{k}(G)$. The image of the map~$\mathrm{co}_k$~generates $D^k(G) / D^{k+1} (G)$; by induction and
 since~$\mathrm{co}_k (a,b)$~is a homomorphism with respect to~$a$~and with respect to~$b$, the
 elements~$[s, x_{k-1}]$~for~$s$~in~$S$~and $x_{k-1} \in X_{k-1}$~generate~$D^k(G) / D^{k+1}(G)$, and
 these elements are exactly the commutators of length~$k$~in the elements of~$S$. The second and third
 assertions follow from the first one.
\end{proof}

\begin{prop}\label{PropSubgroupIsFinitelyGenerated}
  Let~$H$~be a finitely generated nilpotent group, then every subgroup~$H_0$~of~$H$~is finitely generated.
\end{prop}
For a proof see \cite{segal2005polycyclic} where this is actually shown for polycyclic groups, the result follows since
finitely generated nilpotent groups are polycyclic.

  \begin{prop}\label{CorMapNLin}
    Let~$H$~be a nilpotent group of nilpotency class~$t$.
    \begin{enumerate}
      \item the map~$\mathrm{Br}_t : H^t \rightarrow D^{t-1}, (h_1, \cdots, h_t) \mapsto [h_1; h_2;
        \cdots; h_t]$~is multilinear.
      \item If~$\left\{ h_1, \cdots, h_s \right\}$~generates~$H$, then for every~$m \geq 1$, the subgroup generated
        by~$\left\{ h_1^m, \cdots, h_s^m \right\}$~is of finite index in~$H$.
    \end{enumerate}
  \end{prop}

    \begin{proof}[Proof of the first assertion] Let us do an induction on~$t$. The case~$t=1$~being
      trivial, suppose the result true for a
    nilpotent group of class~$t-1$~and consider~$H$~a nilpotent group of class~$t$. Since~$D^t(H) = 0$,
    one has that the map~$\mathrm{co}_{t-1}: (h_1,h) \in H \times D^{t-2}(H) / D^{t-1}(H) \mapsto [h;x] \in
    D^{t-1}(H)$~is bilinear; thus,~$\mathrm{Br}_t$~is a homomorphism with respect to the first factor~$h_1
    \in H$. Let us show that~$\mathrm{Br}_t$~is a homomorphism in the second coordinates~$h_2$, the other
    coordinates are dealt with in the same way. By induction, the map
    \[
      \mathrm{Br}_{t-1}^{H / D^{t-1}(H)}: (H / D^{t-1}(H))^{t-1}\rightarrow D^{t-2}(H) / D^{t-1}(H)
    \]
    is~multilinear. Take~$h_1, h_2, h_2', h_3, \cdots, h_{t-1} \in H$, the multilinearity
    of~$\mathrm{Br}_{t-1}^{H / D^{t-1}(H)}$~provides an element~$g \in D^{t-1}(H)$~such that \[ [h_1; h_2
      h_2'; \cdots ; h_{t-1}] = [h_1, [h_2; \cdots; h_{t-1}]\cdot  [h_2 '; \cdots;
    h_{t-1}] \cdot g] \] and the bilinearity of~$\mathrm{co}_{t-1}$~gives the result since~$[h_1, g] =0$.

  \end{proof}

  \begin{proof}[Proof of the second assertion]
    We set~$S=\left\{ h_1, \cdots, h_s \right\}$~ and we denote
    by~$H_{S,m}$~the subgroup of~$H$~generated by the set ~$\left\{ s^m : s \in S \right\}$. We show by
    induction on~$t=\nilp(H)$~that~$H_{S,m}$~is of finite index in~$H$.

    If~$t=1$~then~$H$~is abelian and
    there is a unique surjective group homomorphism~$\Z^s \rightarrow H$~sending the canonical basis
    to~$S = (h_1, \cdots, h_s)$. The subgroup~$H_{S,m}$~is the image of~$m \Z^s$. Therefore, there is a
    surjective group homomorphism~$\Z^s / m \Z^s \twoheadrightarrow H / H_{S,m}$~and we get that~$H / H_{S,m}$~has at
    most~$m^s$~elements.

    Now suppose the result true for a group of nilpotency class~$t-1$~and assume~$\nilp (H) = t$, with~$t
    \geq 2$.  Set~$T := D^{t-1}(H)$,~$T$~is central in~$H$. One has the exact sequence
    \[ 1 \rightarrow T \rightarrow H \rightarrow H/ T \rightarrow 1. \]

    By induction, the image of~$H_{S,m}$~in~$H / T$~is of finite index; thus, one can fix a finite set~$A
    \subset H$~such that~$H = \bigsqcup_{h \in A} h H_{S,m} T$. To conclude, we only need to show that the
    index of~$T \cap H_{S,m}$~in~$T$~is finite. Since,~$T \cap H_{S,m}$~contains the subgroup
    of~$t-1$~commutators~$D^{t-1}(H_{S,m})$~it suffices to show that the index of~$D^{t-1}(H_{S,m})$
    in~$T$~is finite.

    By Proposition \ref{PropGeneratorsOfLowerCentralSeries},~$T$~is generated by the set~$S' = \left\{ [x_1; \cdots;
      x_{t-1}]: x_i \in S \right\}$~and~$D^{t-1}(H_{S,m})$~is generated by the set~$S'' = \left\{[x_1^m; \cdots;
      x_{t-1}^m] : x_i \in S \right\}$~furthermore, the first assertion shows that~$S''$~consists exactly of the
      elements of~$S'$~raised to the power~$m^{t-1}$. So by the abelian case,~$D^{t-1}(H_{S,m})$~is of finite index
      in~$T$.

  \end{proof}

  \subsubsection{Malcev's completion of nilpotent torsion-free finitely generated group}
  Denote by~$\hat \Z = \prod_{p \text{ prime}} \Z_p$~equipped with the product topology (the adelic topology). It is the
  profinite completion of~$\Z$.
  Let~$H$~be a nilpotent torsion-free finitely generated group. It is known that~$H$~embeds into
 ~$\Tri_1(n, \Z)$~the group of upper triangular matrices with integer coefficients and 1's on the diagonal for some
  integer~$n$~(see for example \cite{segal2005polycyclic} Theorem 2 of Chapter 5). For the rest of this section, we fix an
  embedding~$\iota: H \hookrightarrow \Tri_1(n, \Z)$.  There are two topologies that one can consider on~$\iota(H)$.
  First the adelic topology induced by the inclusion~$\Tri_1(n, \Z) \subset \Tri_1(n, \hat \Z)$, and second, the profinite
  topology where a basis of neighbourhood for the neutral element are the subgroups of finite index in~$\iota(H)$.

  \begin{prop}\label{PropAdelicTopologyAndProfiniteTopologyAreTheSame}
    Let~$G \subset \Tri_1(n, \Z)$~be a subgroup of matrices with integer coefficients and 1's on the diagonal, then the
    profinite topology and the adelic topology on~$G$~are the same. In particular, the profinite completion of $G$
    coincides with the closure of $G$ in $\Tri_1 (n , \hat \Z)$.
  \end{prop}

  \begin{proof}
    First, let~$K$~be a subgroup of~$\GL_n(\Z)$~of the form~$K =\left\{ A \in \GL_n(\Z) : A \equiv \id \mod m \right\}$~for some
    integer~$m$, such groups~$K$~form a basis of open neighbourhood of~$\id$~for the adelic topology. It is a normal
    subgroup of~$\GL_n(\Z)$~with finite quotient, therefore~$G \cap K$~is a finite index
    subgroup of~$G$. Therefore the adelic topology is finer than the profinite topology.

    Conversely,~$G$~is a unipotent group of matrices over~$\Q$,
    therefore it is arithmetic (see \cite{segal2005polycyclic} Exercise 13 of Chapter 6). By the affirmative solution to
    the congruence subgroup problem for arithmetic soluble groups (see \cite{chahal1980}), we get that~$G$~is a
      congruence subgroup. This means that every finite index subgroup of~$G$
      contains a subgroup of the form~$G \cap \left\{ A \in \GL_n (\Z) : A \equiv \id \mod m \right\}$~for some integer~$m$.
      Therefore, the profinite topology is finer than the adelic topology; thus, they are the same.
    \end{proof}

    A consequence of this proposition is that the profinite completion of~$\iota(H)$~is exactly the closure of
   ~$\iota(H)$~in $\Tri_1(n, \hat \Z)$.

  \begin{prop}\label{PropCompletionIsProP}
    Let~$G$~be a nilpotent subgroup of~$\Tri_1(n, \Z)$. The closure of~$G$~in~$\Tri_1(n, \Z_p)$~is the
    pro-$p$-completion of~$G$, in particular it is a~$p$-adic Lie group.
  \end{prop}

  \begin{proof}
    Denote by~$\hat G$~the profinite completion of~$G$~and for a prime~$\ell$,~$G_\ell$~the pro-$\ell$-completion of~$G$.
    Since~$G$~is nilpotent and a finite nilpotent group is a product of~$\ell$-groups for some primes~$\ell$~(see
    \cite{bourbaki1970algebre} chapter 1, \S 7, Theorem~4) we have that~$\hat G = \prod_\ell G_\ell$. By Proposition
    \ref{PropAdelicTopologyAndProfiniteTopologyAreTheSame}, we have a continuous injective homomorphism of topological
    groups

    \[
      \hat G = \prod_\ell G_\ell \hookrightarrow \Tri_1(n, \hat \Z) = \prod_\ell \Tri_1(n, \Z_\ell).
    \]

    For a prime~$p$, this induces a continuous group homomorphism~$G_p \hookrightarrow \prod_\ell \Tri_1(n, \Z_\ell)$.
    But,~$G_p$~is a pro-$p$-group and for every prime~$\ell$,~$\Tri_1(n, \Z_\ell) = \varprojlim \Tri_1(n, \Z /\ell^k
    \Z)$~is a pro-$\ell$-group.  Therefore,~$G_p$~can be identified with the image of~$\hat G$~in~$\Tri_1(n, \Z_p)$;
    this is exactly the completion of~$G$~in~$\Tri_1(n, \Z_p)$, meaning that~$G_p$~is a closed subgroup of the~$p$-adic
    Lie group~$\Tri_1(n, \Z_p)$, so it is a Lie group by  Theorem \ref{theoremClosedSubgroupAreLieGroups}.
  \end{proof}

  \begin{thm}\label{BigtheoremPropClosureIsALieGroup}\label{MinorationpAdic2}
    Let $c>0$ be such that $c > \frac{1}{p-1}$ and let~$H$~be a finitely generated nilpotent subgroup
    of~$\Diff^{an}_c(\Z_p^d)$, then the closure~$\bar H$~of ~$H$~in~$\Diff^{an}(\Z_p^d)$~is a
    finite-dimensional nilpotent Lie group.

    Furthermore, denote by~$\h$~the Lie algebra of~$\bar H$, then~$\h$~is a finite-dimensional nilpotent Lie algebra and
    ~$\dl(\h) \geq \vdl(H)$.
  \end{thm}

  \begin{proof}
    Set~$G = \iota(H)$~and~$\psi := \inv \iota: G \rightarrow \Diff^{an}_c(\Z_p^d)$. By Proposition \ref{PropCompletionIsProP}
    and Proposition \ref{PropEmbeddingOfProPcompletion},~$\psi$
    extends to a Lie group homomorphism~$\psi: G_p \rightarrow \Diff^{an}_c(\Z_p^d)$~where~$G_p$~is the closure
    of~$G$~in~$\Tri_1 (n, \Z_p)$; we show that the image of~$\psi$~is the closure of~$H$~in~$\Diff^{an}(\Z_p^d)$.

    Let~$K$~be the image of~$\psi$.  Since~$\Tri_1(n,\Z_p)$~is compact and~$G_p$~is closed,~$G_p$~is also compact and
    so is~$K$. This implies that the closure~$\overline H$~of~$H$~is included in~$K$. And~$K$~is
    included in~$\overline H$~because of the continuity of~$\psi$. This shows that~$\overline H$~is a finite
    dimensional Lie group isomorphic to~$G_p / \ker \psi$.

    Now, we show the statement for~$\h$. By Proposition
    \ref{PropOpenSubgroupDerivedSeries}, there exists an open subgroup~$H_1$~of~$\overline H$, such that~$D^i (H_1)$~is a
    Lie subgroup of~$\overline H$~with Lie algebra~$\mathcal D^i (\h)$.  Since~$H_1$~is open, by Theorem
    \ref{theoremDiffAnIsALieGroup} there exists  an integer~$c >0$~such that~$\Diff^{an}_c (\Z_p^d) \cap H \subset H_1$.
    Take~$f_1, \cdots, f_s$~generators of~$H$. Then by Proposition \ref{CorMapNLin} the
    subgroup~$H'$~generated by the~$f_i^{p^c}$'s is a finite index subgroup of~$H$~and it is included in~$H_1$~by Lemma
    \ref{lemma:PuissanceCongruence}, therefore~$\dl (\h) = \dl(H_1) \geq \dl(H') \geq \vdl(H)$.
  \end{proof}

  \section{Finitely generated nilpotent groups}\label{SecFinitelyGeneratedNilpotentGroups}
  \subsection{Base change from~$\C$~to~$\Z_p$: Good models}

  To prove Theorem \ref{BoundNilpotentGroups}, we shall ultimately apply Theorem \ref{BigtheoremPropClosureIsALieGroup}.
  Thus, we need a method to transfer problems regarding groups of automorphisms defined over~$\C$~to similar problems on
  groups of Tate analytic diffeomorphisms over~$\Z_p$, for certain primes~$p$.

  \begin{thm}[Lech, see \cite{Lech53}] \label{theoremLechPlongementPadique} Let~$\K$~be a finitely generated field over~$\Q$
    and let~$S$~be a finite subset of~$\K$. Then there exists an infinite number of prime numbers~$p$~with an
    embedding~$\K \hookrightarrow \Q_p$~such that all elements of~$S$~are mapped to~$\Z_p$.
  \end{thm}

  Let~$X$~be an irreducible quasiprojective variety over~$\C$~and~$\Gamma$~a finitely generated subgroup of~$\Aut (X_\C)$.

  \begin{itemize}
    \item Let~$R$~be an integral domain. We say that~$(X,\Gamma)$~is \emph{defined over }~$R$, if there
      exists an irreducible separated reduced scheme~$X_R$~over~$R$~and an injective homomorphism~$\Gamma \hookrightarrow \Aut_\R (X_\R)$
      such that~$X$~and~$\Gamma$~are obtained by the base change~$X = X_R \times_{\Spec R} \Spec \C$.
    \item Let~$p$~be a prime number. A \emph{model} of~$(X,\Gamma)$~over~$\Z_p$~is the data of
      \begin{enumerate}[label=(\roman*)]
        \item A ring~$R \subset \C$~over which~$(X, \Gamma)$~is defined and an embedding~$R \hookrightarrow \Z_p$.
        \item An irreducible variety~$\Chi$~over~$\Z_p$~and an injective homomorphism~$\rho: \Gamma \hookrightarrow \Aut_{\Z_p}
          (\Chi)$~such that \[ \Chi \simeq X_R \times_{\Spec R} \Spec \Z_p. \] is the base change of~$X_R$~and for all~$f
          \in \Gamma$, ~$\rho (f)$~is the base change of~$f$.
      \end{enumerate}
    \item A \emph{good model} over~$\Z_p$~of~$(X,\Gamma)$~is the
      data of a model of~$(X,\Gamma)$~with the additional condition that the special fiber~$\Chi_{\F_p} = \Chi
      \times_{\Spec \Z_p} \Spec \F_p$~is geometrically reduced and irreducible and of dimension \[
      \dim_{\F_p} (\Chi_{\F_p}) = \dim_{\Q_p} (\Chi \times_{\Spec R} \Spec \Q_p). \]
  \end{itemize}

  \begin{prop}[Proposition 4.4 of \cite{bell2010dynamical}, Proposition 3.2 of \cite{cantat2014algebraic}]\label{FromCtoZp}
    Let~$X$~be an irreducible complex quasi-projective variety,~$\alpha \in X(\C)$~and~$\Gamma$~be a finitely generated
    subgroup of~$\Aut_\C (X)$. Then, there exists an infinite number of primes~$p \geq 3$~such that~$(X,\Gamma)$~has a
    good model~$\Chi$~over~$\Z_p$~and such that~$\alpha$~extends to a section~$\alpha: \Spec \Z_p \rightarrow \Chi$.
  \end{prop}

  \begin{ex}
    For simplicity, suppose~$X$~is the affine space~$~\A^d_\C$~with its standard coordinates~$x_1,\cdots,
    x_d$~and~$\Gamma \subset \Aut(\A^d_\C)$~is a finitely generated group of polynomial automorphisms. This is already
    an interesting example. Let~$S$~be a finite symmetrical~$(\inv S = S)$~set of generators of~$\Gamma$. Let~$R$~be the
    ring generated by all the coefficients of the elements of~$S$~and the coordinates of~$\alpha$. Then,~$(X,
    \Gamma)$~is defined over~$R$. Plus, by Theorem \ref{theoremLechPlongementPadique} there exists a prime~$p$~and an
    embedding~$\iota: R \hookrightarrow \Z_p$. Using this embedding, the base change~$\Chi = \A^d_{\Z_p}$~and~$\rho:
    \Gamma \hookrightarrow \Aut (\A^d_{\Z_p})$~show that~$(\A^d, \Gamma)$~is a good model over~$\Z_p$~and~$\alpha$~extends to a
   ~$\Z_p$-point of~$\Chi$.
  \end{ex}

  \subsection{From algebraic automorphisms to analytic diffeomorphisms over~$\Z_p$}
  In this section, we consider a scheme
 ~$\Chi$~of dimension~$d$~over~$\Z_p$, where~$p\geq 3$~is a prime number, such that
  \begin{itemize}
    \item~$\Chi$~is a quasi-projective variety over~$\Z_p$, and its generic fiber is geometrically
      irreducible over ~$\Q_p$.
    \item~$\overline \Chi = \Chi \times_{\Spec \Z_p} \Spec \F_p$~is the special fiber of~$\Chi$~and is geometrically
      irreducible over~$\F_p$.
    \item~$f: \Chi \rightarrow \Chi$~is an automorphism of~$\Z_p$-schemes.
    \item~$\overline f : \overline \Chi \rightarrow \overline \Chi$~is the restriction of~$\Chi$~to the special fiber.
    \item~$r: \Chi (\Z_p) \rightarrow \overline{\Chi} (\F_p)$~is the reduction map.
    \item~$x$~is a smooth~$\F_p$-point and there exists~$\alpha \in \Chi(\Z_p)$~such that~$r(\alpha) = x$.
  \end{itemize}

  For the two next propositions, we refer to \cite{bell2010dynamical}. They will enable us to go from
  algebraic automorphisms to analytic diffeomorphisms.

  \begin{prop}\label{PropExistenceOfIota}
    Let~$\Chi$~be a quasi-projective scheme over~$\Z_p$. There exists a function~$\iota: \Z_p^d \rightarrow
    \Chi(\Z_p)$~which induces an analytic bijection between~$\Z_p^d$~and the open subset of~$\Chi (\Z_p)$~consisting of
    the points~$\beta$~such that~$r(\beta) = x$.
  \end{prop}

  \begin{prop}\label{PropConjugaisonDiffeoAnalytique}
    Suppose that~$\bar f (x) = x$. Let~$\iota: \Z_p^d \rightarrow
    \Chi (\Z_p)$~be the function defined
    in Proposition \ref{PropExistenceOfIota}. Then there exist analytic functions~$F_1,\cdots,F_d \in \Z_p \langle T_1,\cdots,
    T_d \rangle$~such that \begin{enumerate}[label = (\roman*)] \item One has \[ \inv \iota \circ f \circ \iota =
        (F_1,\cdots,F_d) =: \mathcal F \in \Z_p \langle T_1,\cdots,T_d \rangle^d. \]
      \item if~$\bar{\mathcal F}$~is the reduction mod~$p$~of~$\mathcal F$, then~$\bar{\mathcal F} = \mathcal F_0 + \mathcal
        F_1$~with~$\mathcal F_0 \in (\Z/p\Z)^d$~and~$\mathcal F_1 \in
        \GL_d(\Z / p \Z)$.
    \end{enumerate}
    Furthermore~$\mathcal F$~is a Tate-analytic diffeomorphism because~$f$~is an
    automorphism.
  \end{prop}

  \begin{ex}
    Propositions \ref{PropExistenceOfIota} and \ref{PropConjugaisonDiffeoAnalytique} are proven in
    \cite{bell2010dynamical}. We only do the proof in the case~$\Chi = \A^d_{\Z_p}$. Take standard coordinates~$\x = x_1,
    \cdots, x_d$~over~$\Chi$. Then,~$\Chi = \Spec \Z_p[\x]$~and~$\overline \Chi = \Spec \F_p[\x]$. The reduction map~$r:
    \Chi(\Z_p) = \Z_p^d \rightarrow \overline \Chi(\F_p) = \F_p^d$~is the reduction mod~$p$~coordinates by coordinates.

    Take~$x \in \F_p^d$~and~$z \in \Z_p^d$~such that~$r(z) = x$, then the open subset of~$\Chi(\Z_p)$~of elements~$\beta$
    such that~$r(\beta) =x$~is the ball of center~$z$~and radius~$1 / p$. The analytic bijection~$\iota$~is given by
    ~$\iota: m \in \Z_p^d \mapsto z + p \cdot m \in \Chi(\Z_p) = \Z_p^d$. This proves Proposition \ref{PropExistenceOfIota}.

    Now, take a polynomial automorphism~$f$, the map~$\overline f$~is the polynomial automorphism over~$\F_p^d$~obtained
    when taking the coefficients of~$f \mod p$. Take a point~$x \in \F_p^d$~such that~$\bar f (x) = x$, up to a
    conjugation by a translation (which does not change the result), we can suppose that~$x = 0 \in \F_p^d$. This means
    that~$f$~preserves the ball of center
    0 and radius~$1/p$~in~$\Z_p^d$. Writing~$f$~in coordinates, we have
    \[
      f(\x) = p a_0 + A_1 (\x) + A_2(\x) + \cdots
    \]
    where~$a_0 \in \Z_p^d$~and~$A_i$~is the homogeneous part of degree~$i$~of~$f$. Then,
    \[ \inv \iota \circ f \circ \iota (\x) = \frac{1}{p} f(p \x) = a_0 + A_1(\x) + \sum_{k \geq 2} p^{k-1} A_k (\x). \]
    This is indeed an element of~$\Z_p \langle \x \rangle^d$~and~$\overline{\frac{1}{p}f(p \x)}$~is an invertible affine
    transformation of~$\F_p^d$, this proves Proposition \ref{PropConjugaisonDiffeoAnalytique}.
  \end{ex}

  \begin{prop}\label{FromAlgAutoToDIffAnal}[Proposition 3.3 of \cite{cantat2014algebraic}]
    Let~$\Gamma$~be a finitely generated subgroup of~$\Aut_{\Z_p} (\Chi)$. There
    exists a finite index subgroup~$\Gamma_0 \subset \Gamma$~and an open subset~$\mathcal U \subset \Chi (\Z_p)$
    analytically diffeomorphic to~$\Z_p^d$~such that~$\mathcal U$~is stable by the action of~$\Gamma_0$~on~$\Chi$~and this
    action over~$\mathcal U$~is conjugated to the action of a subgroup of~$\Diff^{an}_1 (\mathcal U)$.
  \end{prop}

  \begin{proof}
    Since~$r(\alpha) =x \in \overline \Chi (\F_p)$, the set~$\overline \Chi(\F_p)$~is not empty and since~$\overline \Chi$~has finitely
    many~$\F_p$-points, there exists a finite index subgroup~$\Gamma_1
    \subset \Gamma$~that acts trivially on~$\overline \Chi (\F_p)$. The point~$x$~is fixed by~$\Gamma_1$,
    let~$\iota$~be as in Proposition \ref{PropExistenceOfIota} and~$\mathcal U$~the open subset of~$\Chi(\Z_p)$
    consisting of the points~$\beta$~such that~$r(\beta) = x$.
    Therefore,~$\Gamma_1$~preserves~$\mathcal U$~and by applying Proposition
    \ref{PropConjugaisonDiffeoAnalytique} to the elements of~$\Gamma_1$, we get that conjugation by~$\iota$~induces a
    group homomorphism~$\Gamma_1 \hookrightarrow \Diff^{an}(\Z_p^d)$. Composing this embedding with the homomorphism of
    reduction$\mod p$~induces a group homomorphism from~$\Gamma_1$~to the finite group of affine transformations
    of~$(\Z / p \Z)^d$. Denote by~$\Gamma_0$~the kernel of this homomorphism and the theorem is proven.
  \end{proof}

  \subsection{Proof of Theorem \ref{BoundNilpotentGroups}}

  Take~$H$~a finitely generated nilpotent group acting by algebraic automorphisms on a quasi-projective
  variety~$X$~over a field of characteristic zero.

  We are first going to show that we can suppose~$X$~to be irreducible in order to work on a~$\Z_p$-scheme:~$X$~has a
  finite number of irreducible components and~$H$~permutes them. So there exists a finite index subgroup~$H' \subset H$
  that stabilizes every irreducible component~$X_i$~of~$X$. Call~$H_i$~the restriction of~$H'$~to~$X_i$, then~$H' = \prod
  H_i$~and~$\vdl (H') = \min \vdl (H_i)$. We replace~$X$~by one of its irreducible component of maximal
  dimension and~$H$~by~$H'$~restricted to this component,~$H'$~is also finitely generated by Proposition
  \ref{PropSubgroupIsFinitelyGenerated}.

  Let~$\alpha \in X(\C)$,~$X$~is then an irreducible complex quasi-projective variety of dimension~$d$, by proposition
  \ref{FromCtoZp}, there exists a prime number~$p \geq 3$~such that~$(X,H)$~admits a good model~$\Chi$~over~$\Z_p$~and
  such that~$\alpha$~extends to a~$\Z_p$-point of~$\Chi$. Now,
  by Proposition \ref{FromAlgAutoToDIffAnal}, there exists a finite index subgroup~$H_0 \subset H$~which is isomorphic
  to a subgroup of~$\Diff_1^{an} (\mathcal U)$, for~$\mathcal U$~an open subset of~$\Chi(\Z_p)$~analytically
  diffeomorphic to~$\Z_p^d$. By Proposition \ref{PropSubgroupIsFinitelyGenerated},~$H_0$~is a finitely generated
  nilpotent subgroup of~$\Diff^{an}_1(\Z_p^d)$. Using Theorem \ref{MinorationpAdic2}, we get that the Lie
  algebra~$\h$~associated to~${H_0}$~is nilpotent and~$\dl(\h) \geq \vdl(H_0) \geq \vdl(H)$. Applying Theorem
  \ref{theoremPAdicEpsteinThurston}, we get~$d \geq \vdl(H)$.

  \subsection{Optimality of Theorem \ref{BoundNilpotentGroups}}\label{SubSecOptimality}

  \paragraph{An example from \cite{epstein1979transformation}.--} We will use the construction from
  \cite{epstein1979transformation} to find groups where Theorem \ref{BoundNilpotentGroups} is optimal.

  Let~$n$~be an integer and let~$A$~be the matrix such that~$A(e_i) = e_{i+1}, 1 <i \leq n$~where~$e_i$~is
  the canonical basis. Consider the subgroup of affine transformations~$G = \left\{ x \in \R^n \mapsto \exp(tA) x + b : t \in \R, b \in \R^n \right\}$, we will write~$(t;b)$~for the element~$(x \mapsto \exp (tA) x +
   b)$. This is a real Lie group of dimension~$n+1$~of nilpotency class~$n$~and derived length 2, diffeomorphic to
  ~$\R^{n+1}$. The group law is given by
   \[ (t;b) (s;c) = (t+s; b + e^{tA}c). \]
   Notice that the group law is given by polynomials with rational coefficients in~$s,t$~and the coordinates of~$b$~and
  ~$c$; thus~$G$~is in fact an algebraic group.

   \begin{lemme}\label{LemmeCrochetPolynomial}
     Recall the notation of \ref{SubSubSecNilpotentGroups}. Let~$k < n$~be an integer. The map
     \[ \left( (t_0;b_0), \cdots, (t_k; b_k) \right) \in G^{k+1} = \R^{(n+1)(k+1)} \mapsto \mathrm{Br}_{k+1} \left(
     (t_0; b_0), \cdots, (t_k; b_k) \right) \in G = \R^{n+1} \]
     is a nonconstant polynomial map with rational coefficients from~$\R^{(n+1)(k+1)}$~to~$~\R^{n+1}$.

   \end{lemme}

   \begin{proof}
     The map is polynomial with rational coefficients because the group law is, and this map is
     not constant because~$\nilp (H) = n > k$.
   \end{proof}

   Consider the vector space generated by the translations~$T_{e_i}, 2 \leq i \leq n$.  The Lie
   group~$S$~acts on the variety~$G$~on the left and~$G / S$~is a variety diffeomorphic to~$\R^2$.  The
   diffeomorphisms are given by

  \[ [(t;b)] \in G / S \mapsto (t, b_1) \in \R^2 \]
  and
  \[ (x,y) \in \R^2 \mapsto \left[ (x; y e_1) \right] \in G / S
 \]
 where the brackets mean that we take the orbit under the action of~$S$.

 The group~$G$~acts by right composition on~$G /S$~and this action is faithful. The formulas are given by
 \[ \forall (t;b) \in G, \forall (x,y) \in \R^2 = G / S , \quad (x,y) \cdot (t; b) = \left(x+ t, y +
 \sum_{k=1}^n \frac{t^{k-1}}{(k-1)!} b_k \right). \]
 We see that the action is therefore by polynomial automorphisms. We will write $(t;b)$ on the left even
 though the action is on the right because we view it as a polynomial automorphism of $\A^2_\C$.
 \paragraph{A group where theorem \ref{BoundNilpotentGroups} is optimal.--}

 Now, take~$H$~a finitely generated
 subgroup of~$G$~such that~$\nilp(H) = n$~and~$H$~contains two elements~$(t;b), (s;c)$~such that~$t,s$~and all the
 coordinates of~$b,c$~are algebraically independent over~$\Q$. The group~$H$~satisfies the condition of
 Theorem \ref{BoundNilpotentGroups}, it acts faithfully on the quasiprojective variety~$\A^2_\C$~and we
 have~$\vdl(H) =2$. Indeed, if~$H$~admits an abelian finite index subgroup, then there exists an integer
~$N$~such that~$(t;b)^N$~and~$(s;c)^N$~commute. But this would give a non-trivial polynomial relation over
~$\Q$~between~$s,t$~and the coordinates of~$b,c$~by Lemma \ref{LemmeCrochetPolynomial}, this is absurd. Thus, the bound
 in Theorem \ref{BoundNilpotentGroups} is optimal for~$H$.

 \paragraph{Derived length versus nilpotency class.--} In Theorem \ref{BoundNilpotentGroups} we suppose  that~$H$~is
 nilpotent. One might wonder if the bound can be improved using the virtual nilpotency class, i.e the minimum of~$\nilp
 (H')$~for~$H'$~of finite index in~$H$. We
 show that this is not possible with a similar counterexample as above. Take~$H$~a finitely generated subgroup of~$G$
 such that~$H$~contains~$(t_0; b_0), \cdots, (t_{n-1}; b_{n-1}) \in G^{n}$~such that all the~$t_i$'s and
 the coordinates of the $b_i$'s are algebraically independent over~$\Q$. We show that every finite index subgroup~$H'$~of~$H$~has a nilpotency
 class equal to~$n$. Indeed, there exists an integer~$N$~such that for all~$0 \leq i \leq n-1, h_i := (t_i; b_i)^N \in
 H'$. The coordinates of the~$h_i$'s are still algebraically independent over~$\Q$~because the group law is given by
 polynomials with rational coefficients and by Lemma \ref{LemmeCrochetPolynomial}, the bracket~$[h_0; \cdots;
 h_{n-1}]$~of length~$n$~is not the identity, because that would give a nontrivial polynomial relation between the
 coordinates of the~$h_i$'s.

 \paragraph{Optimality of Theorem \ref{theoremPAdicEpsteinThurston}.--}
  We show that in Theorem \ref{theoremPAdicEpsteinThurston} we can't replace the derived length with the
  nilpotency class and that the theorem is optimal. In fact, the counterexample of
  \cite{epstein1979transformation} can be adapted over~$\Z_p$~as follows. Consider the group~$G$~given by
  \[ G := \left\{ \x \in \Z_p^n \mapsto \exp (p \cdot t A) \x + b : t \in \Z_p, b \in \Z_p^n \right\}. \]
  The group law is now given by polynomials with coefficients in~$\Z_p$~and Lemma \ref{LemmeCrochetPolynomial} still
  holds but the polynomials are with coefficients in~$\Z_p$.

  Then,~$G /S$~is analytically diffeomorphic to~$\Z_p^2$~and we have an embedding of Lie groups~$G \hookrightarrow
  \Diff^{an} (\Z_p^2)$~given by
  \[ \forall (t;b) \in G, \quad (t;b) (x,y) = \left(x + t, y + \sum_{k=1}^n \frac{p^{k-1}t^{k-1}}{(k-1)!} b_k \right). \]
    Let~$\g \subset \Theta (\Z_p^2)$~be the Lie algebra of~$G$,~$\g$~is nilpotent and we show that~$\nilp
    (\g) = n$. Let~$k = \nilp (\g)$, then by Proposition \ref{PropOpenSubgroupDerivedSeries}, there exists
    a small subgroup~$G'$~of~$G$~which is a neighbourhood of $\id$ such that~$\nilp (G') = k$. Therefore~$k \leq n$,
    suppose~$k < n$. By Lemma \ref{LemmeCrochetPolynomial} the map
    \[ (t_0; b_0), \cdots, (t_k; b_k), (x,y) \in \Z_p^{(n+1)(k+1)} \times \Z_p^2 \mapsto \mathrm{Br}_{k+1} ( (t_0;
    b_0), \cdots, (t_k; b_k)) (x,y) \in \Z_p^2 \]
    is polynomial. Let~$P_1 (\mathbf w), P_2(\mathbf w)$~be the first and second coordinate of this map
    where~$\mathbf w$~is a multivariate variable representing all the variables~$t_i, b_i, x,y$. Since,~$\nilp (G) > k$, the
    polynomials~$Q_1(\mathbf w) = P_1 (\mathbf w)- x$, $Q_2 (\mathbf w)= P_2 (\mathbf w)- y$~are not zero. Notice that if
   $(t;b) \in G$, then the Gauss norm of~$(t;b) - \id \in \Z_p \langle x, y \rangle^2$~is bounded by the norm of~$(t;b)
    \in \Z_p^{n+1}$, therefore there exists an integer~$N >0$~such that for all~$(t;b) \in G, (p^N t; p^N b) \in G'$;
    thus \[ Q_1 (p^N \mathbf w) \equiv 0, \quad Q_2 (p^N \mathbf w) \equiv 0 \] and this implies that~$Q_1 = 0, Q_2 =
    0$, this is a contradiction.

    By a similar argument, we can show there are no small abelian subgroups~$G' \subset G$
    neighbourhood of the identity therefore~$\dl(\g) = 2$~by Proposition \ref{PropOpenSubgroupDerivedSeries} and Theorem
    \ref{theoremPAdicEpsteinThurston} is also optimal.

\paragraph*{Acknowledgements.--}
I would like to thank my advisor Serge Cantat for his help. He gave me helpful advice whenever I needed
them. I would also like to thank Junyi Xie for his suggestions. Finally, I would like to thank the reviewer for his/her very
useful observations and detailed advice.

\bibliographystyle{alpha}
\bibliography{ref}

\end{document}